\documentclass[reqno,12pt,letterpaper]{amsart}
\usepackage[proof]{sdmacros}

\newtheorem*{conj*}{Conjecture}

\title{Spectral gaps without the pressure condition}
\author{Jean Bourgain}
\email{bourgain@math.ias.edu}
\address{Institute for Advanced Study, Princeton, NJ 08540}
\author{Semyon Dyatlov}
\email{dyatlov@math.mit.edu}
\address{Department of Mathematics, Massachusetts Institute of Technology,
77 Massachusetts Ave, Cambridge, MA 02139}
\begin{document}

\begin{abstract}
For all convex co-compact hyperbolic surfaces,
we prove the existence of an essential spectral gap,
that is a strip beyond the unitarity axis in which
the Selberg zeta function has only finitely many zeroes.
We make no assumption on the dimension $\delta$ of the limit set,
in particular we do not require the pressure condition $\delta\leq {1\over 2}$.
This is the first result of this kind
for quantum Hamiltonians.

Our proof follows the strategy developed by Dyatlov and Zahl.
The main new ingredient is the fractal uncertainty
principle for $\delta$-regular sets with $\delta<1$,
which may be of independent interest. 
\end{abstract}

\maketitle

%%%%%%%%%%%%%%%%%%%%%%%%%%%%%%%%%%%%%%%%%%%%%%%%%%%%%%%%%%%%%%%%%%%%%%%%%%%%%%%%
%                                 INTRODUCTION                                 %
%%%%%%%%%%%%%%%%%%%%%%%%%%%%%%%%%%%%%%%%%%%%%%%%%%%%%%%%%%%%%%%%%%%%%%%%%%%%%%%%
\addtocounter{section}{1}
\addcontentsline{toc}{section}{1. Introduction}

Let $M=\Gamma\backslash\mathbb H^2$ be a (noncompact) convex co-compact hyperbolic
surface. The Selberg zeta function $Z_M(s)$
is a product over the set $\mathcal L_M$ of all primitive closed
geodesics
$$
Z_M(s)=\prod_{\ell\in\mathcal L_M} \prod_{k=0}^\infty \big(1-e^{-(s+k)\ell}\big),\quad
\Re s\gg 1,
$$
and extends meromorphically to $s\in\mathbb C$. From the spectral description of
$Z_M$ it is known that $Z_M(s)$ has only finitely many zeroes in
$\{\Re s>{1\over 2}\}$, which correspond to small eigenvalues of the Laplacian.
The situation in $\{\Re s\leq {1\over 2}\}$ is more complicated since
the zeroes of $Z_M$ are no longer given by a self-adjoint spectral problem on $L^2(M)$;
they instead correspond to scattering resonances of $M$
and are related to decay of waves.

A natural question is if $M$ has an \emph{essential spectral gap},
that is does there exist $\beta>0$ such that
$Z_M(s)$ has only finitely many zeroes in $\{\Re s>{1\over 2}-\beta\}$?
The known answers so far depend on the exponent of convergence
of the Poincar\'e series of the group, denoted $\delta\in [0,1)$.
Patterson~\cite{Patterson3} and Sullivan~\cite{Sullivan} 
proved that there is a gap of size $\beta={1\over 2}-\delta$
when $\delta<{1\over 2}$ and Naud~\cite{NaudGap}
showed there is a gap of size $\beta>{1\over 2}-\delta$
when $0<\delta\leq {1\over 2}$.
The present paper removes the restrictions on $\delta$:
%%%%%%%%%%%%%%%%%%%%%%%%%%%%%%%%%%%%%%%%%%%%%%%%%%%%%%%%%%%%%%%%%%%%%%%%%%%%%%%%
\begin{theo}
  \label{t:gap}
Every convex co-compact hyperbolic surface has an essential spectral gap.
\end{theo}
%%%%%%%%%%%%%%%%%%%%%%%%%%%%%%%%%%%%%%%%%%%%%%%%%%%%%%%%%%%%%%%%%%%%%%%%%%%%%%%%
Spectral gaps for hyperbolic surfaces have
many important applications,
such as diophantine problems (see Bourgain--Gamburd--Sarnak~\cite{BGS}, Oh--Winter~\cite{OhWinter},
Magee--Oh--Winter~\cite{MOW}, and the
lecture notes by Sarnak~\cite{SarnakThin})
and remainders in the prime geodesic theorem (see for instance the book of Borthwick~\cite[\S14.6]{BorthwickBook}). Moreover, hyperbolic surfaces
are a standard model for more general open quantum chaotic systems,
where spectral gaps have been studied since the work of 
Lax--Phillips~\cite{Lax-Phillips67}, Ikawa~\cite{Ikawa}
and Gaspard--Rice~\cite{GaspardRice}~-- see~\S\ref{s:general-trapping} below. 

Theorem~\ref{t:gap} can also be viewed in terms of the scattering resolvent
%%%%%%%%%%%%%%%%%%%%%%%%%%%%%%%%%%%%%%%%%%%%%%%%%%%%%%%%%%%%%%%%%%%%%%%%%%%%%%%%
$$
R(\lambda)=\Big(-\Delta_M-{1\over 4}-\lambda^2\Big)^{-1}:
\begin{cases}
L^2(M)\to H^2(M),& \Im\lambda>0;\\
L^2_{\comp}(M)\to H^2_{\loc}(M),& \lambda\in\mathbb C.
\end{cases}
$$
%%%%%%%%%%%%%%%%%%%%%%%%%%%%%%%%%%%%%%%%%%%%%%%%%%%%%%%%%%%%%%%%%%%%%%%%%%%%%%%%
where $\Delta_M\leq 0$ is the Laplace--Beltrami operator of $M$.
The family $R(\lambda)$ is meromorphic, as proved by Mazzeo--Melrose~\cite{MazzeoMelrose},
Guillop\'e--Zworski~\cite{GuillopeZworskiAA}, and Guillarmou~\cite{GuillarmouAH}.
Its poles, called \emph{resonances}, correspond to
the zeroes of $Z_M(s)$, $s:={1\over 2}-i\lambda$, see for instance~\cite[Chapter~10]{BorthwickBook}.
Therefore, Theorem~\ref{t:gap} says that
there are only finitely many resonances with $\Im\lambda > -\beta$.
Since
our proof uses~\cite{hgap} and a fractal
uncertainty principle (Theorem~\ref{t:special-fup}),
we obtain a polynomial resolvent bound:
%%%%%%%%%%%%%%%%%%%%%%%%%%%%%%%%%%%%%%%%%%%%%%%%%%%%%%%%%%%%%%%%%%%%%%%%%%%%%%%%
\begin{theo}
  \label{t:resolvent}
Let $M$ be as in Theorem~\ref{t:gap} and take
$\beta=\beta(M)>0$ given by Theorem~\ref{t:special-fup} below. Then for each
$\varepsilon>0$ there exists $C_0>0$ such that for all $\varphi\in C_0^\infty(M)$
\begin{equation}
  \label{e:resolvent-bound}
\|\varphi R(\lambda)\varphi\|_{L^2\to L^2}\leq C |\lambda|^{-1-2\min(0,\Im\lambda)+\varepsilon},\quad
|\lambda|>C_0,\
\Im\lambda\in [-\beta+\varepsilon,1]
\end{equation} 
where the constant $C$ depends on $\varepsilon,\varphi$, but not on $\lambda$.
\end{theo}
%%%%%%%%%%%%%%%%%%%%%%%%%%%%%%%%%%%%%%%%%%%%%%%%%%%%%%%%%%%%%%%%%%%%%%%%%%%%%%%%
\Remarks 1. We see from Theorem~\ref{t:resolvent} that there is
an essential spectral gap of size~$\beta$ for
all $\beta<\beta(M)$ where $\beta(M)$ is given by Theorem~\ref{t:special-fup},
but not necessarily for $\beta=\beta(M)$. However, this is irrelevant
since Theorem~\ref{t:special-fup} does not specify the value of $\beta(M)$.

\noindent 2. Spectral gaps for convex co-compact hyperbolic
surfaces were studied numerically by Borthwick~\cite{BorthwickNum} and Borthwick--Weich~\cite{Borthwick-Weich},
see also~\cite[\S16.3.2]{BorthwickBook} and Figure~\ref{f:numerics}.
%%%%%%%%%%%%%%%%%%%%%%%%%%%%%%%%%%%%%%%%%%%%%%%%%%%%%%%%%%%%%%%%%%%%%%%%%%%%%%%%
\begin{figure}
\includegraphics{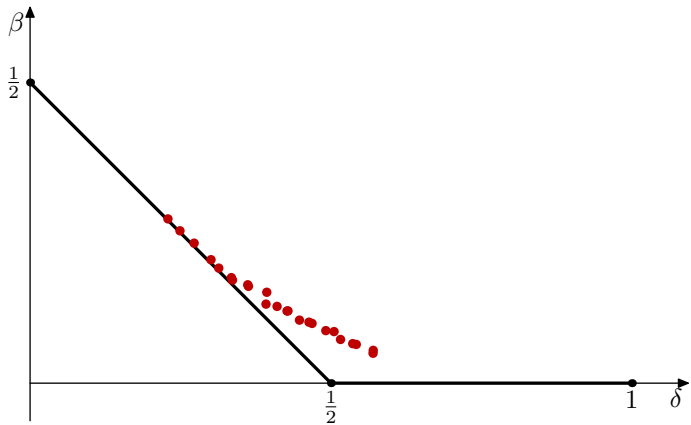}
\caption{Numerically computed essential spectral gaps $\beta$ for symmetric 3-funneled and 4-funneled surfaces from~\cite[Figure~14]{Borthwick-Weich} (specifically, $G_{100}^{I_1}$ in the notation
of~\cite{Borthwick-Weich}; data used with permission of authors). Each point corresponds to one surface and has
coordinates $(\delta,\beta)$. The solid line is the standard gap $\beta=\max(0,{1\over 2}-\delta)$.
}
\label{f:numerics}
\end{figure}
%%%%%%%%%%%%%%%%%%%%%%%%%%%%%%%%%%%%%%%%%%%%%%%%%%%%%%%%%%%%%%%%%%%%%%%%%%%%%%%%

%%%%%%%%%%%%%%%%%%%%%%%%%%%%%%%%%%%%%%%%%%%%%%%%%%%%%%%%%%%%%%%%%%%%%%%%%%%%%%%%
\subsection{Systems with hyperbolic trapping}
  \label{s:general-trapping}

The case of convex co-compact hyperbolic surfaces studied here
belongs to the more general class of \emph{open systems with uniformly
hyperbolic trapped sets which have fractal structure}~-- see the reviews of
Nonnenmacher~\cite{Nonnenmacher} and Zworski~\cite{ZworskiReview} for the definition of these systems and an
overview of the history of the spectral gap problem.
Another example of such systems is given by scattering
in the exterior of several convex obstacles under the no-eclipse
condition, where spectral gaps were 
studied by Ikawa~\cite{Ikawa}, Gaspard--Rice~\cite{GaspardRice},
and Petkov--Stoyanov~\cite{PetkovStoyanov} and observed
experimentally by Barkhofen et al.~\cite{ZworskiPRL}.

For general hyperbolic systems,
resolvent bounds of type~\eqref{e:resolvent-bound}
have important applications to dispersive partial differential equations, including
(the list of references below is by no means extensive)
%%%%%%%%%%%%%%%%%%%%%%%%%%%%%%%%%%%%%%%%%%%%%%%%%%%%%%%%%%%%%%%%%%%%%%%%%%%%%%%%
\begin{itemize}
\item exponential local energy decay $\mathcal O(e^{-\beta t})$
of linear waves
modulo a finite dimensional space corresponding to resonances
with $\Im\lambda > -\beta$, see Christianson~\cite{Christianson} and Guillarmou--Naud~\cite{GuillarmouNaudDecay};
\item exponential stability for nonlinear wave equations,
see Hintz--Vasy \cite{HintzVasyGreat};
\item local smoothing estimates, see Datchev~\cite{Kiril};
\item Strichartz estimates, see Burq--Guillarmou--Hassell~\cite{BGH} and Wang~\cite{WangJian}.
\end{itemize}
%%%%%%%%%%%%%%%%%%%%%%%%%%%%%%%%%%%%%%%%%%%%%%%%%%%%%%%%%%%%%%%%%%%%%%%%%%%%%%%%
Theorem~\ref{t:resolvent} is the first unconditional
spectral gap result for quantum chaotic Hamiltonians
with fractal hyperbolic trapped sets. It is a step towards
the following general spectral gap conjecture:
%%%%%%%%%%%%%%%%%%%%%%%%%%%%%%%%%%%%%%%%%%%%%%%%%%%%%%%%%%%%%%%%%%%%%%%%%%%%%%%%
\begin{conj*}
\cite[\S3.2, Conjecture 3]{ZworskiReview}
Suppose that $P$ is an operator for which the scattering resolvent admits a meromorphic
continuation (e.g. $P=-\Delta_M$ where $(M,g)$ is a complete Riemannian
manifold with Euclidean or asymptotically hyperbolic infinite ends).
Assume that the underlying classical flow (e.g. the geodesic flow on $(M,g)$)
has a compact hyperbolic trapped set.

Then there exists $\beta>0$ such that $(M,g)$ has an essential spectral
gap of size $\beta$, that is the scattering resolvent has only finitely
many poles with $\Im\lambda>-\beta$.
\end{conj*}
%%%%%%%%%%%%%%%%%%%%%%%%%%%%%%%%%%%%%%%%%%%%%%%%%%%%%%%%%%%%%%%%%%%%%%%%%%%%%%%%
We give a brief overview of some of the previous works related to this conjecture,
as well as some recent results.
We remark that the question of which scattering systems have exponential
wave decay has been studied since the work of Lax--Phillips~\cite{Lax-Phillips67},
see~\cite[Epilogue]{Lax-Phillips89} for an overview of the history of this question.
%%%%%%%%%%%%%%%%%%%%%%%%%%%%%%%%%%%%%%%%%%%%%%%%%%%%%%%%%%%%%%%%%%%%%%%%%%%%%%%%
\begin{itemize}
\item A spectral gap of size $\beta=-P({1\over 2})$ under
the \emph{pressure condition} $P({1\over 2})<0$ was proved
for obstacle scattering by Ikawa~\cite{Ikawa},
computed in the physics literature by Gaspard--Rice~\cite{GaspardRice},
and proved for general hyperbolic trapped sets by Nonnenmacher--Zworski~\cite{NonnenmacherZworskiActa,NonnenmacherZworskiAMRX}. Here $P(\sigma)$ is the \emph{topological
pressure} of the system, see for instance~\cite[(14)]{Nonnenmacher}
or~\cite[(3.28)]{ZworskiReview}. For the case
of hyperbolic surfaces considered here,
we have $P(\sigma)=\delta-\sigma$, so the pressure condition is
$\delta<{1\over 2}$ and the pressure gap is the Patterson--Sullivan gap.
\item
An improved spectral gap $\beta>-P({1\over 2})$ under the relaxed
pressure condition $P({1\over 2})\leq 0$ was proved
by Naud~\cite{NaudGap} for convex co-compact hyperbolic surfaces,
Stoyanov~\cite{Stoyanov1} for more general cases of Ruelle zeta functions including
higher-dimensional convex co-compact hyperbolic manifolds,
and Petkov--Stoyanov~\cite{PetkovStoyanov}
for obstacle scattering. The above papers rely
on the method originally developed by Dolgopyat~\cite{Dolgopyat}.
\item Jakobson--Naud~\cite{Jakobson-Naud2} conjectured a gap of size $-{1\over 2}P(1)={1-\delta\over 2}$
for hyperbolic surfaces and obtained upper bounds on the size of the gap.
\item Dyatlov--Zahl~\cite{hgap} reduced the spectral gap question
for convex co-compact hyperbolic manifolds to a \emph{fractal uncertainty principle}
(see~\S\ref{s:intro-hyperbolic} below) and showed
an improved gap $\beta>{1\over 2}-\delta$ for surfaces with $\delta={1\over 2}$
and for nearby surfaces using methods from additive combinatorics. The size of the gap in~\cite{hgap}
decays superpolynomially as a function
of the regularity constant $C_R$ (defined in~\S\ref{s:fup-regular} below).
Dyatlov--Jin~\cite{regfup} adapted the methods of~\cite{Dolgopyat,NaudGap}
to obtain an improved gap for $0<\delta\leq {1\over 2}$ which
depends polynomially on $C_R$. Later Bourgain--Dyatlov~\cite{hyperfup}
gave an improved gap $\beta>{1\over 2}-\delta$ which depends only on $\delta>0$ and not
on $C_R$. The present paper is in some sense orthogonal to~\cite{regfup,hyperfup}
since it gives a gap $\beta>0$; thus the result of the present paper is interesting
when $\delta\geq{1\over 2}$ and of~\cite{regfup,hyperfup}, when $\delta\leq {1\over 2}$.
\item In a related setting of open quantum baker's maps,
Dyatlov--Jin~\cite{oqm} used a fractal uncertainty
principle to show that every such system has a gap, and obtained
quantitative bounds on the size of the gap.
\item We finally discuss the case of scattering
on finite area hyperbolic surfaces with cusps. An
example is the modular surface $\PSL(2,\mathbb Z)\backslash\mathbb H^2$, where zeroes
of the Selberg zeta function fall into two categories:
%%%%%%%%%%%%%%%%%%%%%%%%%%%%%%%%%%%%%%%%%%%%%%%%%%%%%%%%%%%%%%%%%%%%%%%%%%%%%%%%
\begin{enumerate}
\item infinitely many \emph{embedded eigenvalues} on the line $\{\Re s={1\over 2}\}$;
\item the rest, corresponding to the zeroes
of the Riemann zeta function.
\end{enumerate}
%%%%%%%%%%%%%%%%%%%%%%%%%%%%%%%%%%%%%%%%%%%%%%%%%%%%%%%%%%%%%%%%%%%%%%%%%%%%%%%%
In particular the modular surface has no essential spectral gap.
Same is true for any finite area surface, that is there are infinitely
many resonances in the half-plane $\{\Re s>{1\over 2}-\beta\}$ for all $\beta>0$.
This follows from the fact that
the number of resonances in a ball of radius $T$ grows
like $T^2$, together with the following bound proved by Selberg~\cite[Theorem~1]{Selberg}:
$$
\sum_{s\text{ resonance}\atop
|\Im s|\leq T} \Big({1\over 2}-\Re s\Big)=\mathcal O(T\log T)\quad\text{as }T\to\infty.
$$
However, the question of how close resonances
can lie to the critical line $\Re s={1\over 2}$ for a \emph{generic}
finite area surface is more complicated, in particular embedded eigenvalues
are destroyed by generic conformal perturbations of the metric,
see Colin de Verdi\`ere~\cite{CdV1,CdV2}, and by generic
perturbations within the class of hyperbolic surfaces,
see Phillips--Sarnak~\cite{Phillips-Sarnak}.

Note that the present paper does not
apply to the finite area case for two reasons: (1) the methods
of~\cite{hgap} do not apply to manifolds with cusps,
in particular because the trapped set is not compact,
and (2) finite area surfaces have $\Lambda_\Gamma=\mathbb S^1$
and thus $\delta=1$.
\end{itemize}
%%%%%%%%%%%%%%%%%%%%%%%%%%%%%%%%%%%%%%%%%%%%%%%%%%%%%%%%%%%%%%%%%%%%%%%%%%%%%%%%

%%%%%%%%%%%%%%%%%%%%%%%%%%%%%%%%%%%%%%%%%%%%%%%%%%%%%%%%%%%%%%%%%%%%%%%%%%%%%%%%
\subsection{Uncertainty principle for hyperbolic limit sets}
\label{s:intro-hyperbolic}

The proof of Theorem~\ref{t:gap} uses the strategy
of~\cite{hgap}, which reduced the spectral gap question
to a fractal uncertainty principle. To state it,
define
the operator $\mathcal B_\chi=\mathcal B_\chi(h)$ on $L^2(\mathbb S^1)$ by
\begin{equation}
  \label{e:B-chi}
\mathcal B_\chi f(y)=(2\pi h)^{-1/2}\int_{\mathbb S^1}|y-y'|^{2i/h}
\chi(y,y')f(y')\,dy'
\end{equation}
where $|y-y'|$ denotes the Euclidean distance on $\mathbb R^2$ restricted
to the unit circle $\mathbb S^1$ and
$$
\chi\in C_0^\infty(\mathbb S^1_\Delta),\quad
\mathbb S^1_\Delta:=\{(y,y')\in\mathbb S^1\times\mathbb S^1\mid
y\neq y'\}.
$$
The semiclassical parameter $h>0$ corresponds to the inverse of the frequency
and also to the inverse of the spectral parameter: $h\sim |\lambda|^{-1}$.
We will be interested in the limit $h\to 0$. The operator $\mathcal B_\chi$
is bounded on $L^2(\mathbb S^1)$ uniformly in $h$, see~\cite[\S5.1]{hgap}. We can view $\mathcal B_\chi$
as a hyperbolic analogue of the (semiclassically rescaled) Fourier transform.

A key object associated to the surface $M$
is the \emph{limit set} $\Lambda_\Gamma\subset\mathbb S^1$;
see for instance~\cite[\S2.2.1]{BorthwickBook}
or~\cite[(4.11)]{hgap} for the definition.
Theorems~\ref{t:gap} and~\ref{t:resolvent}
follow by combining~\cite[Theorem~3]{hgap} with the following
uncertainty principle for $\Lambda_\Gamma$:
%%%%%%%%%%%%%%%%%%%%%%%%%%%%%%%%%%%%%%%%%%%%%%%%%%%%%%%%%%%%%%%%%%%%%%%%%%%%%%%%
\begin{theo}
  \label{t:special-fup}
Let $M=\Gamma\backslash\mathbb H^2$ be a convex co-compact hyperbolic surface
and denote by $\Lambda_\Gamma(h^\rho)\subset\mathbb S^1$ the $h^\rho$-neighborhood of the limit set.
Then
there exist $\beta>0$ and $\rho\in (0,1)$ depending
only on $M$ such that
for all $\chi\in C_0^\infty(\mathbb S^1_\Delta)$
and $h\in (0,1)$
\begin{equation}
  \label{e:special-fup}
\|\indic_{\Lambda_\Gamma(h^\rho)}\mathcal B_{\chi}(h)\indic_{\Lambda_\Gamma(h^\rho)}\|_{L^2(\mathbb S^{1})\to L^2(\mathbb S^1)}\leq Ch^{\beta}
\end{equation}
where the constant $C$ depends on $M,\chi$, but not on $h$.
\end{theo}
%%%%%%%%%%%%%%%%%%%%%%%%%%%%%%%%%%%%%%%%%%%%%%%%%%%%%%%%%%%%%%%%%%%%%%%%%%%%%%%%
\Remarks \noindent 1.
We call~\eqref{e:special-fup} an uncertainty principle
because it implies that no quantum state can be microlocalized $h^\rho$
close to $\Gamma_\pm\subset S^*M$, where $S^*M$ denotes the cosphere bundle
of $M$ and $\Gamma_\pm$ are the incoming/outgoing tails,
consisting of geodesics trapped in the future ($\Gamma_-$) or in the past ($\Gamma_+$).
The lifts of $\Gamma_\pm$ to $S^*\mathbb H^2$ can be expressed in terms of~$\Lambda_\Gamma$. See~\cite[\S\S1.1, 4.1.2]{hgap} for details, in particular for how to define microlocalization
to an $h^\rho$-neighborhood of $\Gamma_\pm$.

\noindent 2. Recent work of Dyatlov--Zworski~\cite{tug} provides an alternative
to~\cite{hgap} for showing that Theorem~\ref{t:special-fup}
implies Theorem~\ref{t:gap}, using transfer operator techniques.

\noindent 3. The value of $\beta$ depends only on $\delta$
and the regularity constant $C_R$ of the set $\Lambda_\Gamma$
(see~\S\S\ref{s:fup-regular},\ref{s:special-fup-proof}).
Recently Jin--Zhang~\cite[Theorem~1.3]{JinZhang} obtained an estimate
on $\beta$ in terms of~$\delta,C_R$ which has the form (here $\mathbf K$ is a large universal constant)
$$
\beta=\exp\Big[-\mathbf K(C_R\delta^{-1}(1-\delta)^{-1})^{\mathbf K(1-\delta)^{-3}}\Big].
$$
The parameter $\rho$ will be taken very close to 1 depending
on $\delta,C_R$, see~\eqref{e:rho-fixed}.

\noindent 4. If we vary $M$ within the moduli
space $\mathscr M$ of convex co-compact hyperbolic surfaces, then $\delta$ changes
continuously (in fact, real analytically). Moreover, as shown in~\cite[Lemma~2.12]{hyperfup},
the regularity constant $C_R$ can be estimated explicitly in terms of the disks
and group elements in a Schottky representation of $M$ and
thus is bounded locally uniformly on $\mathscr M$.
Therefore the value of $\beta$ is bounded away from 0 as long
as $M$ varies in a compact subset of $\mathscr M$.

%%%%%%%%%%%%%%%%%%%%%%%%%%%%%%%%%%%%%%%%%%%%%%%%%%%%%%%%%%%%%%%%%%%%%%%%%%%%%%%%
\subsection{Uncertainty principle for regular fractal sets}
  \label{s:fup-regular}

In order to prove Theorem~\ref{t:special-fup} we exploit the
fractal structure of the limit set $\Lambda_\Gamma$.
For simplicity we make the illegal choice of $\rho:=1$ in the
informal explanations
below.

The (Hausdorff and Minkowski) dimension of $\Lambda_\Gamma$
is equal to $\delta\in [0,1)$, so the volume
of $\Lambda_\Gamma(h)$ decays like $h^{1-\delta}$
as $h\to 0$. For $\delta<{1\over 2}$ this implies
(using the $L^1\to L^\infty$ estimate on $\mathcal B_\chi(h)$
together with H\"older's inequality)
the uncertainty principle~\eqref{e:special-fup} with
$\beta={1\over 2}-\delta$ and thus recovers the Patterson--Sullivan
gap~-- see~\cite[\S5.1]{hgap}.

However, for $\delta\geq {1\over 2}$
one cannot obtain~\eqref{e:special-fup} by using only the volume
of the set $\Lambda_\Gamma(h)$. Indeed, if we replace
$\Lambda_\Gamma(h)$ by an interval of size $h^{1/2}$,
then a counterexample to~\eqref{e:special-fup} is given by a Gaussian
wavepacket of width $h^{1/2}$.
Therefore, one needs to exploit finer fractal structure of the limit set.
For us such structure is given by \emph{Ahlfors--David regularity},
which roughly speaking states that $\Lambda_\Gamma$ has dimension $\delta$
at each point on each scale:
%%%%%%%%%%%%%%%%%%%%%%%%%%%%%%%%%%%%%%%%%%%%%%%%%%%%%%%%%%%%%%%%%%%%%%%%%%%%%%%%
\begin{defi}
  \label{d:regular-set}
Let $X\subset\mathbb R$ be a nonempty closed set and $\delta\in [0,1]$,
$C_R\geq 1$, $0\leq \alpha_0\leq\alpha_1\leq\infty$.
We say that $X$ is \textbf{$\delta$-regular with constant $C_R$
on scales $\alpha_0$ to $\alpha_1$}, if there exists
a Borel measure $\mu_X$ on $\mathbb R$ such that:
\begin{itemize}
\item $\mu_X$ is supported on $X$, that is $\mu_X(\mathbb R\setminus X)=0$;
\item for each interval $I$ of size $|I|\in [\alpha_0,\alpha_1]$,
we have $\mu_X(I)\leq C_R |I|^\delta$;
\item if additionally $I$ is centered at a point in $X$, then
$\mu_X(I)\geq C_R^{-1}|I|^\delta$.
\end{itemize}
\end{defi}
%%%%%%%%%%%%%%%%%%%%%%%%%%%%%%%%%%%%%%%%%%%%%%%%%%%%%%%%%%%%%%%%%%%%%%%%%%%%%%%%
\Remarks 1. The condition that $\mu_X$ is supported on $X$ is never used in this paper
(and the measure $\mu_X$ is referred to explicitly only in~\S\ref{s:regular-sets}),
however we keep it to make the definition compatible with~\cite{regfup}.

\noindent 2. In estimates regarding regular sets, it will be important
that the constants involved may depend on $\delta,C_R$,
but not on $\alpha_0,\alpha_1$. Thus it is useful to think
of $\delta,C_R$ as fixed and $\alpha_1/\alpha_0$ as large.

\noindent 3. As indicated above, the limit set $\Lambda_\Gamma$ is $\delta$-regular
on scales $0$ to $1$ where
$\delta\in [0,1)$ is the exponent
of convergence of the Poincar\'e series of the group
$\Gamma$~-- see~\S\ref{s:special-fup-proof}.

The key component of the proof of Theorem~\ref{t:special-fup} is the following
fractal uncertainty principle for the Fourier transform
and general $\delta$-regular sets; it is a result of independent interest.
In~\S\ref{s:special-fup} we show that Theorem~\ref{t:general-fup}
implies Theorem~\ref{t:special-fup} by linearizing the phase
of the operator~\eqref{e:B-chi}. (This makes the value of the exponent
$\beta$ smaller~-- see the remark following Proposition~\ref{l:fup-fio}.)
%%%%%%%%%%%%%%%%%%%%%%%%%%%%%%%%%%%%%%%%%%%%%%%%%%%%%%%%%%%%%%%%%%%%%%%%%%%%%%%%
\begin{theo}
  \label{t:general-fup}
Let $0\leq\delta<1$, $C_R\geq 1$, $N\geq 1$ and assume that
\begin{itemize}
\item $X\subset [-1, 1]$ is $\delta$-regular
with constant $C_R$ on scales $N^{-1}$ to $1$, and
\item $Y\subset [-N, N]$ is $\delta$-regular
with constant $C_R$ on scales $1$ to $N$.
\end{itemize}
Then there exist $\beta>0,C$ depending only on $\delta,C_R$
such that for all $f\in L^2(\mathbb R)$
\begin{equation}
  \label{e:general-fup}
\supp\hat f\subset Y\quad\Longrightarrow\quad
\|f\|_{L^2(X)}\leq CN^{-\beta}\|f\|_{L^2(\mathbb R)}.
\end{equation}
Here $L^2(X)$ is defined using the Lebesgue measure.
\end{theo}
%%%%%%%%%%%%%%%%%%%%%%%%%%%%%%%%%%%%%%%%%%%%%%%%%%%%%%%%%%%%%%%%%%%%%%%%%%%%%%%%
\Remark
Since $X$ is only required to be $\delta$-regular down to scale $N^{-1}$,
rather than $0$, it may contain intervals of size $N^{-1}$ and thus have
positive Lebesgue measure. In fact it is useful to picture $X$ as a union
of intervals of size $N^{-1}$ distributed in a fractal way,
and similarly picture $Y$ as a union of intervals of size 1.
See also Lemma~\ref{l:regular-fatten}.

The proof of Theorem~\ref{t:general-fup} is given in~\S\ref{s:general-fup}.
We give here a brief outline. The key component is the following
nonstandard quantitative unique continuation result, Proposition~\ref{l:step}:
for each $c_1>0$ there exists $c_3>0$ depending only on $\delta,C_R,c_1$
such that
\begin{equation}
  \label{e:intro-step}
f\in L^2(\mathbb R),\quad
\supp\hat f\subset Y
\quad\Longrightarrow\quad
\|f\|_{L^2(U')}\geq c_3\|f\|_{L^2(\mathbb R)}
\end{equation}
where $Y$ is as in Theorem~\ref{t:general-fup}
and $U'=\bigcup_{j\in\mathbb Z}I'_j$
where each $I'_j\subset [j,j+1]$ is an (arbitrarily chosen) subinterval of size $c_1$.
It is important that $c_3$, as well as other constants in the argument, does not depend on the large parameter $N$.

Theorem~\ref{t:general-fup} follows from~\eqref{e:intro-step} by iteration on scale.
Here $\delta$-regularity of~$X$ with $\delta<1$ is used to obtain the missing subinterval
property (see Lemmas~\ref{l:missing-interval} and~\ref{l:missing-child}):
there exists $c_1=c_1(\delta,C_R)>0$ such that for all $j\in\mathbb Z$, the set
$[j,j+1]\setminus X$ contains some interval $I'_j$ of size $c_1$,
and same is true for dilates $\alpha X$ when $1\leq \alpha\ll N$.
The lower bound~\eqref{e:intro-step} gives an upper bound on the $L^2$ norm of $f$ on $\mathbb R\setminus U'\supset X$,
which iterated $\sim\log N$ times gives the power improvement in~\eqref{e:general-fup}.
See~\S\ref{s:iteration-argument} for details.

To prove~\eqref{e:intro-step}, we first show a similar bound where
the support condition on $\hat f$ is replaced by a decay condition:
for $\theta(\xi):=\log(10+|\xi|)^{-{1+\delta\over 2}}$ and all $f\in L^2(\mathbb R)$
\begin{equation}
  \label{e:intro-interpol}
\big\|\exp\big(\theta(\xi)|\xi|\big)\hat f(\xi)\big\|_{L^2(\mathbb R)}
\leq C_1 \|f\|_{L^2(\mathbb R)}
\quad\Longrightarrow\quad
\|f\|_{L^2(U')}\geq c_3\|f\|_{L^2(\mathbb R)}
\end{equation}
where $c_3$ depends only on $\delta,c_1,C_1$.
The proof uses estimates on harmonic measures for domains
of the form $\{|\Im z|<r\}\setminus I'_j\subset\mathbb C$.
See the remark following the statement of Lemma~\ref{l:harmest-main}.

Coming back to~\eqref{e:intro-step}, we construct a function
$\psi\not\equiv 0$ which is compactly supported,
more precisely $\supp\psi\subset [-{c_1\over 10},{c_1\over 10}]$, and
 satisfies the Fourier decay bound
\begin{equation}
  \label{e:fordec}
|\widehat\psi(\xi)|\leq \exp\big(-c_2 \theta(\xi)|\xi|\big)\quad\text{for all }\xi\in Y
\end{equation}
where $c_2>0$ depends only on $\delta,C_R,c_1$.
To do that, we use $\delta$-regularity of $Y$ with $\delta<1$ to construct
a weight $\omega:\mathbb R\to (0,1]$ such that
$$
\begin{gathered}
\sup |\partial_\xi\log \omega|\leq C_0,\quad
\int_{\mathbb R}{|\log\omega(\xi)|\over 1+\xi^2}\,d\xi\leq C_0,\\
\omega(\xi)\leq \exp\big(-\theta(\xi)|\xi|\big)\quad\text{for all }\xi\in Y
\end{gathered}
$$
where $C_0$ depends only on $\delta,C_R$.
By a quantitative version of the Beurling--Malliavin
Multiplier Theorem (see Lemma~\ref{l:bmmt}) there exists $\psi\not\equiv 0$ with
the required support property
and $|\widehat\psi(\xi)|\leq \omega(\xi)^{c_2}$ for all $\xi\in\mathbb R$,
thus~\eqref{e:fordec} holds.
See~\S\ref{s:adapted-multiplier} for details.

Finally, we put
$$
g:=f*\psi\in L^2(\mathbb R),\quad
\hat g(\xi)=\hat f(\xi)\widehat\psi(\xi).
$$
If $\supp\hat f\subset Y$, then by~\eqref{e:fordec}
we have
\begin{equation}
  \label{e:forv}
\big\|\exp\big(c_2\theta(\xi)|\xi|\big)\hat g(\xi)\big\|_{L^2}\leq \|f\|_{L^2}.
\end{equation}
On the other hand if $U'':=\bigcup_{j\in\mathbb Z}I''_j$ where
$I''_j\subset I'_j$ is the interval with the same center as $I'_j$ and size $c_1/2$,
then the support condition on $\psi$ implies that
$g=(\indic_{U'}f)*\psi$ on $U''$ and thus 
$\|g\|_{L^2(U'')}\leq \|f\|_{L^2(U')}$.
We revise the proof of~\eqref{e:intro-interpol}
with $f$ replaced by $g$, $U'$ by $U''$,
and the Fourier decay bound replaced by~\eqref{e:forv},
to obtain~\eqref{e:intro-step} and thus finish the proof of Theorem~\ref{t:general-fup}.
In the process we apply the argument with $Y$ replaced by its translates
$Y+\ell$, $\ell\in\mathbb Z$, $|\ell|\leq N$; to each translate corresponds
its own multiplier $\psi$. See~\S\ref{s:step} for details.

%%%%%%%%%%%%%%%%%%%%%%%%%%%%%%%%%%%%%%%%%%%%%%%%%%%%%%%%%%%%%%%%%%%%%%%%%%%%%%%%
%%%%%%%%%%%%%%%%%%%%%%%%%%%%%%%%%%%%%%%%%%%%%%%%%%%%%%%%%%%%%%%%%%%%%%%%%%%%%%%%
\section{Preliminaries}

%%%%%%%%%%%%%%%%%%%%%%%%%%%%%%%%%%%%%%%%%%%%%%%%%%%%%%%%%%%%%%%%%%%%%%%%%%%%%%%%
\subsection{Notation}

We first introduce the notation used in the paper.

For two sets $X,Y\subset \mathbb R$,
define
$X+Y:=\{x+y\mid x\in X,\ y\in Y\}$.
For $\lambda\geq 0$, denote $\lambda X:=\{\lambda x\mid x\in X\}$.
For an interval $I=x_0+[-r,r]\subset\mathbb R$ with $r\geq 0$,
denote by $|I|:=2r$ the size of~$I$ and
say that $x_0$ is the center of~$I$.
For $X\subset\mathbb R$ and $\alpha\geq 0$, define the
$\alpha$-neighborhood of $X$ by
\begin{equation}
  \label{e:nbhd}
X(\alpha):=X+[-\alpha,\alpha]\ \subset\ \mathbb R.
\end{equation}
For $X\subset\mathbb R$, denote by $\mathbf 1_X\in L^\infty(\mathbb R)$
the indicator function of $X$ and by
$\indic_X:L^2(\mathbb R)\to L^2(\mathbb R)$ the corresponding
multiplication operator. For each $\xi\in\mathbb R$, denote
$$
\langle\xi\rangle := \sqrt{1+|\xi|^2}.
$$
We use the following convention for the Fourier transform
of $f\in L^1(\mathbb R)$:
\begin{equation}
  \label{e:fourier}
\hat f(\xi)=\mathcal Ff(\xi)=\int_{\mathbb R}e^{-2\pi ix\xi}f(x)\,dx.
\end{equation}
One advantage of this convention is that $\mathcal F$
extends to a unitary operator on $L^2(\mathbb R)$.
Recall the Fourier inversion formula
\begin{equation}
  \label{e:inverse-fourier}
f(x)=\mathcal F^*\hat f(x)=\int_{\mathbb R}e^{2\pi ix\xi}\hat f(\xi)\,d\xi
\end{equation}
and the convolution formula
\begin{equation}
  \label{e:convolution-fourier}
\widehat{f*g}(\xi)=\hat f(\xi)\cdot \hat g(\xi).
\end{equation}
For $s\in\mathbb R$, define the Sobolev space $H^s(\mathbb R)$ with the norm
\begin{equation}
  \label{e:sobolev-space}
\|f\|_{H^s}:=\|\langle\xi\rangle^s \hat f(\xi)\|_{L^2}.
\end{equation}
We also use the unitary semiclassical Fourier transform $\mathcal F_h$
on $L^2(\mathbb R)$ defined by
\begin{equation}
  \label{e:F-h}
\mathcal F_h f(\xi)=h^{-1/2}\int_{\mathbb R} e^{-2\pi i x\xi/h}f(x)\,dx
=h^{-1/2}\hat f\Big({\xi\over h}\Big),\quad
h>0.
\end{equation}
The following identity holds for all $X,Y\subset\mathbb R$, $x_0,y_0\in\mathbb R$,
and $h$, and follows directly from the fact that $\mathcal F_h^*$ conjugates shifts
to multiplication operators:
\begin{equation}
  \label{e:shifted-fup}
\|\indic_{X+x_0} \mathcal F_h^* \indic_{Y+y_0}\|_{L^2\to L^2}
=\|\indic_X\mathcal F_h^*\indic_Y\|_{L^2\to L^2}.
\end{equation}
We also note the following corollary of the triangle inequality:
\begin{equation}
  \label{e:fup-splitting}
X\subset \bigcup_j X_j,\quad
Y\subset \bigcup_k Y_k\quad\Rightarrow\quad
\|\indic_X\mathcal F_h^*\indic_Y\|_{L^2\to L^2}
\leq \sum_{j,k}\|\indic_{X_j}\mathcal F_h^*\indic_{Y_k}\|_{L^2\to L^2}.
\end{equation}
Finally, we record the following version of H\"older's inequality:
\begin{equation}
  \label{e:holder-special}
\sum_j a_j^\kappa\cdot b_j^{1-\kappa}\leq \Big(\sum_j a_j\Big)^\kappa\cdot
\Big(\sum_j b_j\Big)^{1-\kappa},\quad
a_j,b_j\geq 0,\
\kappa\in (0,1).
\end{equation}

%%%%%%%%%%%%%%%%%%%%%%%%%%%%%%%%%%%%%%%%%%%%%%%%%%%%%%%%%%%%%%%%%%%%%%%%%%%%%%%%
\subsection{Regular sets}
  \label{s:regular-sets}

We now establish properties of $\delta$-regular sets (see Definition~\ref{d:regular-set}),
some of which have previously appeared in~\cite{hgap}. For the reader's convenience
we first give a few examples:
%%%%%%%%%%%%%%%%%%%%%%%%%%%%%%%%%%%%%%%%%%%%%%%%%%%%%%%%%%%%%%%%%%%%%%%%%%%%%%%%
\begin{itemize}
\item $\{0\}$ is $0$-regular on scales $0$ to~$\infty$ with constant $C_R=1$;
\item $[0,1]$ is $1$-regular on scales $0$ to~$1$ with constant $2$;
\item the mid-third Cantor set $\mathcal C\subset [0,1]$
is $\log_2 3$-regular on scales $0$ to $1$ with constant $100$,
see~\cite[\S5.2]{regfup} for examples of more general Cantor sets;
\item the set $[0,1]\sqcup \{2\}$ cannot be $\delta$-regular on scales $0$ to~$1$
with any constant for any $\delta$;
\item the set $[0,h^{1/2}]$ cannot be $\delta$-regular on scales $h$ to~$1$
with any $h$-independent constant for any $\delta$ (here $0<h\ll 1$).
\end{itemize}
%%%%%%%%%%%%%%%%%%%%%%%%%%%%%%%%%%%%%%%%%%%%%%%%%%%%%%%%%%%%%%%%%%%%%%%%%%%%%%%%

We next show that certain
operations preserve the class of $\delta$-regular sets if we allow to increase
the regularity constant and shrink the scales
on which regularity is imposed.
The precise dependence of the new regularity constant
on the original one, though specified in the lemmas below, is not important
for our later proofs.
%%%%%%%%%%%%%%%%%%%%%%%%%%%%%%%%%%%%%%%%%%%%%%%%%%%%%%%%%%%%%%%%%%%%%%%%%%%%%%%%
\begin{lemm}[Affine transformations]
  \label{l:regular-scale}
Let $X$ be a $\delta$-regular set with constant $C_R$
on scales $\alpha_0$ to $\alpha_1$. Fix $\lambda>0$ and $y\in\mathbb R$.
Then the set
$\widetilde X:=y+\lambda X$ is $\delta$-regular with constant $C_R$
on scales $\lambda\alpha_0$ to $\lambda\alpha_1$.
\end{lemm}
%%%%%%%%%%%%%%%%%%%%%%%%%%%%%%%%%%%%%%%%%%%%%%%%%%%%%%%%%%%%%%%%%%%%%%%%%%%%%%%%
\begin{proof}
This is straightforward to verify, taking the measure
$$
\mu_{\widetilde X}(A):=\lambda^\delta\mu_X\big(\lambda^{-1}(A-y)\big).\qedhere
$$
\end{proof}
%%%%%%%%%%%%%%%%%%%%%%%%%%%%%%%%%%%%%%%%%%%%%%%%%%%%%%%%%%%%%%%%%%%%%%%%%%%%%%%%
%
%%%%%%%%%%%%%%%%%%%%%%%%%%%%%%%%%%%%%%%%%%%%%%%%%%%%%%%%%%%%%%%%%%%%%%%%%%%%%%%%
\begin{lemm}[Increasing the upper scale]
  \label{l:regular-expand-top}
Let $X$ be a $\delta$-regular set with constant $C_R$
on scales $\alpha_0$ to $\alpha_1$. Fix $T\geq 1$. Then
$X$ is $\delta$-regular with constant $\widetilde C_R:=2TC_R$
on scales $\alpha_0$ to $T\alpha_1$.
\end{lemm}
%%%%%%%%%%%%%%%%%%%%%%%%%%%%%%%%%%%%%%%%%%%%%%%%%%%%%%%%%%%%%%%%%%%%%%%%%%%%%%%%
\begin{proof}
Let $I$ be an interval such that $\alpha_0\leq |I|\leq T\alpha_1$.
We first show the upper bound $\mu_X(I)\leq \widetilde C_R|I|^\delta$.
For $\alpha_0\leq |I|\leq\alpha_1$ this is immediate, so
we may assume that $\alpha_1<|I|\leq T\alpha_1$. Then
$I$ can be covered by $\lceil T\rceil\leq 2T$ intervals
of size $\alpha_1$ each, therefore
$$
\mu_X(I)\ \leq\ 2T\cdot C_R\alpha_1^\delta\ \leq\ \widetilde C_R|I|^\delta.
$$
Now, assume that $I$ is centered at a point in $X$. We show
the lower bound $\mu_X(I)\geq \widetilde C_R^{-1}|I|^\delta$.
As before, we may assume that $\alpha_1<|I|\leq T\alpha_1$.
Let $I'\subset I$ be the interval with the same center and
$|I'|=\alpha_1$. Then
$$
\mu_X(I)\ \geq\ \mu_X(I')\ \geq\ C_R^{-1}\alpha_1^\delta\ \geq\ \widetilde C_R^{-1}|I|^\delta.\qedhere
$$
\end{proof}
%%%%%%%%%%%%%%%%%%%%%%%%%%%%%%%%%%%%%%%%%%%%%%%%%%%%%%%%%%%%%%%%%%%%%%%%%%%%%%%%
%
%%%%%%%%%%%%%%%%%%%%%%%%%%%%%%%%%%%%%%%%%%%%%%%%%%%%%%%%%%%%%%%%%%%%%%%%%%%%%%%%
\begin{lemm}[Neighborhoods]
  \label{l:regular-fatten}
Let $X$ be a $\delta$-regular set with constant
$C_R$ on scales $\alpha_0$ to $\alpha_1\geq 2\alpha_0$. Fix $T\geq 1$. Then
the neighborhood $X(T\alpha_0)=X+[-T\alpha_0,T\alpha_0]$
is $\delta$-regular with constant $\widetilde C_R:=4TC_R$ on scales $2\alpha_0$ to $\alpha_1$.
\end{lemm}
%%%%%%%%%%%%%%%%%%%%%%%%%%%%%%%%%%%%%%%%%%%%%%%%%%%%%%%%%%%%%%%%%%%%%%%%%%%%%%%%
\begin{proof}
Put $\widetilde X:=X(T\alpha_0)$ and define the measure
$\mu_{\widetilde X}$ supported on $\widetilde X$ by convolution:
$$
\mu_{\widetilde X}(A):={1\over T\alpha_0} \int_{-T\alpha_0}^{T\alpha_0}\mu_X(A+y)\,dy.
$$
Let $I$ be an interval such that $2\alpha_0\leq |I|\leq \alpha_1$. Then
$$
\mu_{\widetilde X}(I)\ \leq\ 2C_R |I|^\delta\ \leq\ \widetilde C_R |I|^\delta.
$$
Now, assume additionally that $I$ is centered at a point $x_1\in \widetilde X$.
Take $x_0\in X$ such that $|x_0-x_1|\leq T\alpha_0$
and let $I'$ be the interval of size ${1\over 2}|I|$ centered at $x_0$.
Then $\mu_X(I')\geq (2C_R)^{-1}|I|^\delta$.
Let $J=x_0-x_1+[-{1\over 2}\alpha_0,{1\over 2}\alpha_0]$, then $J\cap [-T\alpha_0,T\alpha_0]$ is an interval
of size at least~${1\over 2}\alpha_0$ and for each $y\in J$, we have $I'\subset I+y$.
It follows that
$$
\mu_{\widetilde X}(I)
\ \geq\ {1\over 2T}\mu_X(I')\ \geq\ \widetilde C_R^{-1}|I|^\delta.\qedhere
$$
\end{proof}
%%%%%%%%%%%%%%%%%%%%%%%%%%%%%%%%%%%%%%%%%%%%%%%%%%%%%%%%%%%%%%%%%%%%%%%%%%%%%%%%
%
%%%%%%%%%%%%%%%%%%%%%%%%%%%%%%%%%%%%%%%%%%%%%%%%%%%%%%%%%%%%%%%%%%%%%%%%%%%%%%%%
\begin{lemm}[Nonlinear transformations]
  \label{l:regular-nonlinear}
Assume that $F:\mathbb R\to\mathbb R$ is a $C^1$ diffeomorphism such that for some constant
$C_F\geq 1$
$$
C_F^{-1}\ \leq \ |\partial_x F|\ \leq\ C_F.
$$
Let $X$ be a $\delta$-regular
set with constant $C_R$ on scales $\alpha_0$ to $\alpha_1\geq C_F^2\alpha_0$.
Then $F(X)$ is a $\delta$-regular set with constant $\widetilde C_R:=C_F C_R$
on scales $C_F\alpha_0$ to $C_F^{-1}\alpha_1$.
\end{lemm}
%%%%%%%%%%%%%%%%%%%%%%%%%%%%%%%%%%%%%%%%%%%%%%%%%%%%%%%%%%%%%%%%%%%%%%%%%%%%%%%%
\begin{proof}
Put $\widetilde X:=F(X)$ and define the measure $\mu_{\widetilde X}$
supported on $\widetilde X$ as a pullback:
$$
\mu_{\widetilde X}(A):=\mu_X(F^{-1}(A)).
$$
Let $\widetilde I$ be an interval with $C_F\alpha_0\leq |\widetilde I|\leq C_F^{-1}\alpha_1$.
Take the interval $I:=F^{-1}(\widetilde I)$.
Then
$$
C_F^{-1}|\widetilde I|\ \leq\ |I|\ \leq\ C_F|\widetilde I|.
$$
In particular, $\alpha_0\leq |I|\leq \alpha_1$. Therefore,
$$
\mu_{\widetilde X}(\widetilde I)\ =\ \mu_X(I)\ \leq\
C_R|I|^\delta\ \leq\ \widetilde C_R|\widetilde I|^\delta. 
$$
If additionally $\widetilde I$ is centered at a point $\tilde x\in \widetilde X$,
then $I$ contains the interval $I'$ of size $C_F^{-1}|\widetilde I|$ centered at $F^{-1}(\tilde x)\in X$.
Therefore,
$$
\mu_{\widetilde X}(\widetilde I)\ \geq\ \mu_X(I')\ \geq\ C_R^{-1}|I'|^\delta
\ \geq\ \widetilde C_R^{-1}|\widetilde I|^\delta.
\qedhere
$$
\end{proof}
%%%%%%%%%%%%%%%%%%%%%%%%%%%%%%%%%%%%%%%%%%%%%%%%%%%%%%%%%%%%%%%%%%%%%%%%%%%%%%%%
%
%%%%%%%%%%%%%%%%%%%%%%%%%%%%%%%%%%%%%%%%%%%%%%%%%%%%%%%%%%%%%%%%%%%%%%%%%%%%%%%%
\begin{lemm}[Intersections with intervals]
  \label{l:regular-intersection}
Let $X$ be a $\delta$-regular set with constant $C_R$
on scales $\alpha_0$ to $\alpha_1$. Fix two different intervals
$J\subset J'$ with the same center and $|J'|-|J|\geq\alpha_0$.
Assume that $X\cap J$ is nonempty and $X\cap J'\subset J$. Then
$X\cap J$ is $\delta$-regular with constant~$C_R$
on scales $\alpha_0$ to $\tilde\alpha_1:=\min(\alpha_1,|J'|-|J|)$.
\end{lemm}
%%%%%%%%%%%%%%%%%%%%%%%%%%%%%%%%%%%%%%%%%%%%%%%%%%%%%%%%%%%%%%%%%%%%%%%%%%%%%%%%
\begin{proof}
Put $\widetilde X:=X\cap J=X\cap J'$ and consider the measure
$\mu_{\widetilde X}(A):=\mu_X(A\cap J')$ supported on~$\widetilde X$.
Let $I$ be an interval with $\alpha_0\leq |I|\leq \tilde\alpha_1$. Then
$$
\mu_{\widetilde X}(I)\ \leq\ \mu_X(I)\ \leq\ C_R|I|^\delta.
$$
Now, assume that $I$ is centered at some $x\in\widetilde X$.
Then $x\in J$ and thus $I\subset J'$, giving
$$
\mu_{\widetilde X}(I)\ =\ \mu_X(I)\ \geq\ C_R^{-1}|I|^\delta.\qedhere
$$
\end{proof}
%%%%%%%%%%%%%%%%%%%%%%%%%%%%%%%%%%%%%%%%%%%%%%%%%%%%%%%%%%%%%%%%%%%%%%%%%%%%%%%%
We now establish further properties of $\delta$-regular sets, starting with
a quantitative version of the fact that
every $\delta$-regular set with $\delta<1$ is nowhere dense:
%%%%%%%%%%%%%%%%%%%%%%%%%%%%%%%%%%%%%%%%%%%%%%%%%%%%%%%%%%%%%%%%%%%%%%%%%%%%%%%%
\begin{lemm}[The missing subinterval property]
  \label{l:missing-interval}
Let $X$ be a $\delta$-regular set with constant $C_R$
on scales $\alpha_0$ to $\alpha_1$, and $0\leq \delta<1$.
Fix an integer
\begin{equation}
  \label{e:missingint}
L\geq(3C_R)^{2\over 1-\delta}.
\end{equation}
Assume that $I$ is an interval with $\alpha_0\leq |I|/L< |I|\leq \alpha_1$
and $I_1,\dots,I_L$ is the partition of $I$ into intervals of size $|I|/L$.
Then there exists $\ell$ such that $X\cap I_\ell=\emptyset$.
\end{lemm}
%%%%%%%%%%%%%%%%%%%%%%%%%%%%%%%%%%%%%%%%%%%%%%%%%%%%%%%%%%%%%%%%%%%%%%%%%%%%%%%%
\begin{proof}
Using Lemma~\ref{l:regular-scale}, we reduce to the case
$I=[0,L]$, $\alpha_0\leq 1<L\leq \alpha_1$.
Then $I_\ell=[\ell-1,\ell]$.
We argue by contradiction, assuming that each $I_\ell$
intersects $X$.
Then $I'_\ell:=[\ell-3/2,\ell+1/2]$
contains a size 1 interval centered at a point in $X$ and thus
$$
\mu_X(I'_\ell)\geq C_R^{-1}\quad\text{for all }\ell=1,\dots,L.
$$
On the other hand, $\bigcup_{\ell=1}^{L}I'_\ell=[-1/2,L+1/2]$
can be covered by 2 intervals of size $L$ and each point lies in at most 3 of the intervals $I'_\ell$.
Therefore,
$$
C_R^{-1} L\ \leq\ \sum_{\ell=1}^L \mu_X(I'_\ell)
\ \leq\ 3\mu_X\Big(\bigcup_{\ell=1}^{L} I'_\ell\Big)
\ \leq\ 6C_RL^\delta
$$
which contradicts~\eqref{e:missingint}.
\end{proof}
%%%%%%%%%%%%%%%%%%%%%%%%%%%%%%%%%%%%%%%%%%%%%%%%%%%%%%%%%%%%%%%%%%%%%%%%%%%%%%%%
We next obtain
the following fact used in~\S\ref{s:pseudo}:
%%%%%%%%%%%%%%%%%%%%%%%%%%%%%%%%%%%%%%%%%%%%%%%%%%%%%%%%%%%%%%%%%%%%%%%%%%%%%%%%
\begin{lemm}[Splitting into smaller regular sets]
  \label{l:regular-split}
Let $X$ be a $\delta$-regular set with constant~$C_R$ on scales
$\alpha_0$ to $\alpha_1$ and assume that $0\leq \delta<1$ and $(4C_R)^{2\over 1-\delta}\alpha_0\leq \rho\leq \alpha_1$.
Then there exists a collection of disjoint intervals $\mathcal J$ such that
\begin{equation}
  \label{e:regular-split}
X=\bigsqcup_{J\in\mathcal J} (X\cap J);\quad
(4C_R)^{-{2\over 1-\delta}}\rho\leq |J|\leq \rho\quad\text{for all }J\in\mathcal J
\end{equation}
and each $X\cap J$ is $\delta$-regular with constant $\widetilde C_R:=(4C_R)^{2\over 1-\delta}C_R$
on scales $\alpha_0$ to $\rho$.
\end{lemm}
%%%%%%%%%%%%%%%%%%%%%%%%%%%%%%%%%%%%%%%%%%%%%%%%%%%%%%%%%%%%%%%%%%%%%%%%%%%%%%%%
\begin{proof}
Fix an integer $L$ satisfying~\eqref{e:missingint}
and $L\leq (4C_R)^{2\over 1-\delta}$.
Consider the intervals
$$
I_\ell:={\rho\over L}[\ell,\ell+1],\quad \ell\in\mathbb Z.
$$
By Lemma~\ref{l:missing-interval}, for each $\ell$
at least one of the intervals $I_\ell,I_{\ell+1},\dots,I_{\ell+L-1}$
does not intersect $X$. Define the collection $\mathcal J$ as follows:
$J\in\mathcal J$ if and only if $J=I_\ell\cup\dots\cup I_r$
for some $\ell\leq r$, each of the intervals
$I_\ell,\dots,I_r$ intersects $X$,
but $I_{\ell-1},I_{r+1}$ do not intersect $X$.
Then~\eqref{e:regular-split} holds.

For each $J=I_\ell\cup\dots\cup I_r\in\mathcal J$, take
$J':=I_{\ell-1}\cup\dots\cup I_{r+1}$. Then $X\cap J'\subset J$
and $|J'|-|J|=2\rho/L$.
By Lemma~\ref{l:regular-intersection},
$X\cap J$ is $\delta$-regular with constant $C_R$
on scales $\alpha_0$ to $2\rho/L$. Then by Lemma~\ref{l:regular-expand-top},
$X\cap J$ is $\delta$-regular with constant $\widetilde C_R$
on scales $\alpha_0$ to $\rho$.
\end{proof}
%%%%%%%%%%%%%%%%%%%%%%%%%%%%%%%%%%%%%%%%%%%%%%%%%%%%%%%%%%%%%%%%%%%%%%%%%%%%%%%%
The following covering statement is
used in the proof of Lemma~\ref{l:special-multiplier}:
%%%%%%%%%%%%%%%%%%%%%%%%%%%%%%%%%%%%%%%%%%%%%%%%%%%%%%%%%%%%%%%%%%%%%%%%%%%%%%%%
\begin{lemm}[The small cover property]
  \label{l:regular-cover}
Let $X$ be a $\delta$-regular set with constant $C_R$ on scales $\alpha_0$
to $\alpha_1$.
Let $I$ be an interval and $\rho>0$ satisfy
$\alpha_0\leq \rho\leq |I|\leq \alpha_1$.
Then there exists a nonoverlapping collection%
\footnote{A collection of intervals is nonoverlapping if the intersection of each two different intervals
is either empty or consists of one point.}
$\mathcal J$ of $N_{\mathcal J}$ intervals of size $\rho$ each such that
$$
X\cap I\ \subset\ \bigcup_{J\in\mathcal J} J,\quad
N_{\mathcal J}\leq 12C_R^2\Big({|I|\over \rho}\Big)^\delta.
$$
\end{lemm}
%%%%%%%%%%%%%%%%%%%%%%%%%%%%%%%%%%%%%%%%%%%%%%%%%%%%%%%%%%%%%%%%%%%%%%%%%%%%%%%%
\begin{proof}
Let $\mathcal J$ consist of all intervals of the form $\rho[j,j+1]$,
$j\in\mathbb Z$ which intersect $X\cap I$. Then $X\cap I\subset \bigcup_{J\in\mathcal J}J$.
It remains to prove the upper bound on $N_{\mathcal J}$.
For this we use an argument similar to the one in Lemma~\ref{l:missing-interval}.

For each $J\in\mathcal J$, let $J'\supset J$ be the interval with the same center
and $|J'|=2\rho$. Since $J$ intersects $X$, $J'$ contains an interval
of size $\rho$ centered at a point in $X$. Therefore,
$$
\mu_X(J')\geq C_R^{-1}\rho^\delta.
$$
On the other hand, $\bigcup_{J\in\mathcal J}J'\subset I({3\over 2}\rho)$ can be covered by 4 intervals
of size $|I|$ and each point lies in at most 3 of the intervals
$J'$. Therefore,
$$
N_{\mathcal J}\cdot C_R^{-1}\rho^\delta
\ \leq\ \sum_{J\in\mathcal J}\mu_X(J')
\ \leq\ 3\mu_X\Big(\bigcup_{J\in\mathcal J}J'\Big)
\ \leq\ 12 C_R|I|^\delta
$$
which implies the upper bound on $N_{\mathcal J}$.
\end{proof}
%%%%%%%%%%%%%%%%%%%%%%%%%%%%%%%%%%%%%%%%%%%%%%%%%%%%%%%%%%%%%%%%%%%%%%%%%%%%%%%%
%
%%%%%%%%%%%%%%%%%%%%%%%%%%%%%%%%%%%%%%%%%%%%%%%%%%%%%%%%%%%%%%%%%%%%%%%%%%%%%%%%
\begin{lemm}[Lebesgue measure of a regular set]
  \label{l:regular-lebesgue}
Let $X\subset [-\alpha_1,\alpha_1]$ be a $\delta$-regular set with constant $C_R$
on scales $\alpha_0>0$ to $\alpha_1$. Then the Lebesgue measure of $X$ satisfies
\begin{equation}
  \label{e:regular-lebesgue}
\mu_L(X)\leq  24C_R^2\alpha_1^\delta\alpha_0^{1-\delta}.
\end{equation}
\end{lemm}
%%%%%%%%%%%%%%%%%%%%%%%%%%%%%%%%%%%%%%%%%%%%%%%%%%%%%%%%%%%%%%%%%%%%%%%%%%%%%%%%
\begin{proof}
Applying Lemma~\ref{l:regular-cover} with $I:=[0,\alpha_1]$, $\rho:=\alpha_0$,
we cover $X\cap I$ with at most $12C_R^2(\alpha_1/\alpha_0)^\delta$
intervals of size $\alpha_0$ each. It follows that
$$
\mu_L(X\cap I)\ \leq\ 12C_R^2\Big({\alpha_1\over\alpha_0}\Big)^\delta\cdot \alpha_0
\ =\ 12C_R^2\alpha_1^\delta\alpha_0^{1-\delta}.
$$
Repeating the argument with $I:=[-\alpha_1,0]$ and combining
the resulting two bounds, we get~\eqref{e:regular-lebesgue}.
\end{proof}
%%%%%%%%%%%%%%%%%%%%%%%%%%%%%%%%%%%%%%%%%%%%%%%%%%%%%%%%%%%%%%%%%%%%%%%%%%%%%%%%
We finally describe a tree discretizing a $\delta$-regular set.
(This tree is simpler than the one used in~\cite{hgap} and~\cite{regfup}
because we do not merge consecutive intervals.)
Let $X\subset\mathbb R$ be a set and fix an integer $L\geq 2$,
the base of the discretization.
Put
\begin{equation}
  \label{e:covering-tree}
V_n(X):=\Big\{I=\Big[{j\over L^n},{j+1\over L^n}\Big]
\,\Big|\ j\in\mathbb Z,\
I\cap X\neq\emptyset\Big\},\quad
n\in\mathbb Z.
\end{equation}
Note that $X\subset \bigcup_{I\in V_n(x)}I$ for all $n$.
Moreover, each $I'\in V_n(X)$ is contained in exactly one $I\in V_{n-1}(X)$;
we say that $I$ is the \emph{parent} of $I'$ and $I'$ is a \emph{child} of $I$.
Each interval has at most $L$ children.

The next lemma, used in~\S\ref{s:iteration-argument},
follows immediately from Lemma~\ref{l:missing-interval}:
%%%%%%%%%%%%%%%%%%%%%%%%%%%%%%%%%%%%%%%%%%%%%%%%%%%%%%%%%%%%%%%%%%%%%%%%%%%%%%%%
\begin{lemm}[Each parent is missing a child]
  \label{l:missing-child}
Let $X$ be a $\delta$-regular set with constant~$C_R$ on scales $\alpha_0$
to $\alpha_1$, and $0\leq \delta<1$. Let $L$ satisfy~\eqref{e:missingint}
and take $n\in\mathbb Z$ such that $\alpha_0\leq L^{-n-1}\leq L^{-n}\leq\alpha_1$.
Then each $I\in V_n(X)$ has at most $L-1$ children.
\end{lemm}
%%%%%%%%%%%%%%%%%%%%%%%%%%%%%%%%%%%%%%%%%%%%%%%%%%%%%%%%%%%%%%%%%%%%%%%%%%%%%%%%

%%%%%%%%%%%%%%%%%%%%%%%%%%%%%%%%%%%%%%%%%%%%%%%%%%%%%%%%%%%%%%%%%%%%%%%%%%%%%%%%
\subsection{The Multiplier Theorem}
  \label{s:bmmt}

We next present the multiplier theorem originally due to Beurling--Malliavin~\cite{Beurling-Malliavin}
which is the key harmonic analysis tool in the proof of Theorem~\ref{t:general-fup}.
It will be used in the proof of Lemma~\ref{l:special-multiplier} below,
with a weight $\omega$ taylored to the fractal set $Y$.
We refer the reader to
Mashreghi--Nazarov--Havin~\cite{Havin} for a discussion of the history of this theorem
and recent results.
%%%%%%%%%%%%%%%%%%%%%%%%%%%%%%%%%%%%%%%%%%%%%%%%%%%%%%%%%%%%%%%%%%%%%%%%%%%%%%%%
\begin{theo}\cite[Theorem~BM1]{Havin}
  \label{t:bmmt-theorem}
Let $\omega\in C^1(\mathbb R;(0,1])$ satisfy the conditions
\begin{align}
  \label{e:bmmt-thm-1}
\int_{\mathbb R} {|\log\omega(\xi)|\over 1+\xi^2}\,d\xi& < \infty,\\
  \label{e:bmmt-thm-2}
\sup|\partial_\xi\log\omega|& < \infty.
\end{align}
Then for each $c_0>0$ there exists a function $\psi\in L^2(\mathbb R)$ such that
\begin{equation}
  \label{e:bmmt-thm}
\supp\psi\subset [-c_0,c_0],\quad
|\widehat\psi|\leq \omega,\quad
\psi\not\equiv 0.
\end{equation}
\end{theo}
%%%%%%%%%%%%%%%%%%%%%%%%%%%%%%%%%%%%%%%%%%%%%%%%%%%%%%%%%%%%%%%%%%%%%%%%%%%%%%%%
\Remark
Condition~\eqref{e:bmmt-thm-1} states that $\omega(\xi)$ does not
come too close to 0 too often as $|\xi|\to\infty$. This condition is necessary
to have a compactly supported $\psi\not\equiv 0$ with $|\widehat\psi|\leq\omega$,
see~\eqref{e:weight-condition}.

We will use in~\S\ref{s:adapted-multiplier}
the following quantitative refinement of Theorem~\ref{t:bmmt-theorem}:
%%%%%%%%%%%%%%%%%%%%%%%%%%%%%%%%%%%%%%%%%%%%%%%%%%%%%%%%%%%%%%%%%%%%%%%%%%%%%%%%
\begin{lemm}
  \label{l:bmmt}
For all $C_0,c_0>0$ there exists $c=c(C_0,c_0)>0$ such that the following holds.
Let $\omega\in C^1(\mathbb R;(0,1])$ be a weight function satisfying
\begin{align}
  \label{e:bmmt-1}
\int_{\mathbb R} {|\log\omega(\xi)|\over 1+\xi^2}\,d\xi& \leq C_0,\\
  \label{e:bmmt-2}
\sup|\partial_\xi\log\omega|& \leq C_0.
\end{align}
Then there exists a function $\psi\in L^2(\mathbb R)$ such that
\begin{equation}
  \label{e:bmmt-conclusion}
\supp \psi\subset [-c_0,c_0],\quad
|\widehat\psi|\leq \omega^{c},\quad
\|\widehat\psi\|_{L^2(-1,1)}\geq c.
\end{equation}
\end{lemm}
%%%%%%%%%%%%%%%%%%%%%%%%%%%%%%%%%%%%%%%%%%%%%%%%%%%%%%%%%%%%%%%%%%%%%%%%%%%%%%%%
\begin{proof}
We argue by contradiction. Fix $C_0,c_0>0$ such that Lemma~\ref{l:bmmt}
does not hold. Then there exists a sequence of weights $\omega_1,\omega_2,\dots$
each satisfying~\eqref{e:bmmt-1}, \eqref{e:bmmt-2} and such that
for each $\psi\in L^2(\mathbb R)$
\begin{equation}
  \label{e:bm-internal}
\supp\psi\subset [-c_0,c_0],\quad
|\widehat\psi|\leq (\omega_n)^{2^{-n}}\quad\Longrightarrow\quad
\|\widehat\psi\|_{L^2(-1,1)}\leq 2^{-n}.
\end{equation}
Define the weight $\omega$ by
$$
\omega:=\prod_{n=1}^\infty (\omega_n)^{2^{-n}}.
$$
Then $\omega$ satisfies~\eqref{e:bmmt-thm-1}, \eqref{e:bmmt-thm-2}.
By Theorem~\ref{t:bmmt-theorem} there exists
$\psi\in L^2(\mathbb R)$ satisfying~\eqref{e:bmmt-thm}.
For each $n$, we have
$|\widehat\psi|\leq \omega\leq (\omega_n)^{2^{-n}}$. Then~\eqref{e:bm-internal}
implies that $\|\widehat\psi\|_{L^2(-1,1)}\leq 2^{-n}$
for all $n$ and thus $\widehat\psi=0$ on $(-1,1)$. However,
since $\psi$ is compactly supported, $\widehat\psi$ is real analytic and thus
$\psi\equiv 0$, which contradicts~\eqref{e:bmmt-thm}.
\end{proof}
%%%%%%%%%%%%%%%%%%%%%%%%%%%%%%%%%%%%%%%%%%%%%%%%%%%%%%%%%%%%%%%%%%%%%%%%%%%%%%%%

%%%%%%%%%%%%%%%%%%%%%%%%%%%%%%%%%%%%%%%%%%%%%%%%%%%%%%%%%%%%%%%%%%%%%%%%%%%%%%%%
\subsection{Harmonic measures on slit domains}
  \label{s:harmonic-measures}

We finally review the facts we need from the theory of harmonic
measures, referring the reader to Conway~\cite[Chapter~21]{Conway},
Aleman--Feldman--Ross~\cite{Aleman}, and It\^o--McKean~\cite[\S7]{Ito-McKean}
for more details. These facts are used in~\S\ref{s:harmonic-fourier-bound} below.

Let $\Omega\subset\mathbb C$ be a bounded open domain with smooth boundary $\partial\Omega$
and take $t\in\Omega$. The \emph{harmonic measure of $\Omega$ centered at $t$},
denoted $\mu_t^\Omega$, is defined as follows. Let $f\in C(\partial\Omega)$
and let $u$ be the harmonic extension of $f$, namely the unique
function such that $u\in C(\overline\Omega)$,
$u|_{\partial\Omega}=f$, and $u$ is harmonic in $\Omega$. Then
for all $f$, we have
$$
u(t)=\int_{\partial\Omega} f\,d\mu_t^\Omega.
$$
Such $\mu_t^\Omega$ is a (nonnegative) probability measure;
indeed, nonnegativity follows from the maximum principle and
$\mu_t^\Omega(\partial\Omega)=1$ since $1$ is a harmonic function.
Moreover, since $\partial\Omega$ is smooth,
$\mu_t^\Omega$ is absolutely continuous with respect
to the arclength measure $\mu_L$ on $\partial\Omega$.
We denote by ${d\mu_t^\Omega\over d\mu_L}$ the
corresponding Radon--Nikodym derivative.

Since harmonic functions are invariant under conformal transformations,
we have the following fact: if $\Omega'$
is another bounded domain with smooth boundary
and $\varkappa:\Omega\to\Omega'$
is a conformal transformation extending to a homeomorphism $\overline\Omega\to\overline\Omega'$, then
\begin{equation}
  \label{e:harmonic-conformal}
\mu_t^\Omega(A)=\mu_{\varkappa(t)}^{\Omega'}(\varkappa(A))\quad\text{for all }t\in\Omega,\ A\subset\partial\Omega.
\end{equation}
Another interpretation of harmonic measure is as follows (see for instance~\cite[\S7.12]{Ito-McKean}): let $(W_\tau)_{\tau\geq 0}$ be the Brownian motion starting at $W_0=t$. Then
$\mu_t^\Omega$ is the probability distribution of the point on $\partial\Omega$ through which
$W_\tau$ exits $\Omega$ first.
%%%%%%%%%%%%%%%%%%%%%%%%%%%%%%%%%%%%%%%%%%%%%%%%%%%%%%%%%%%%%%%%%%%%%%%%%%%%%%%%
\begin{figure}
\includegraphics{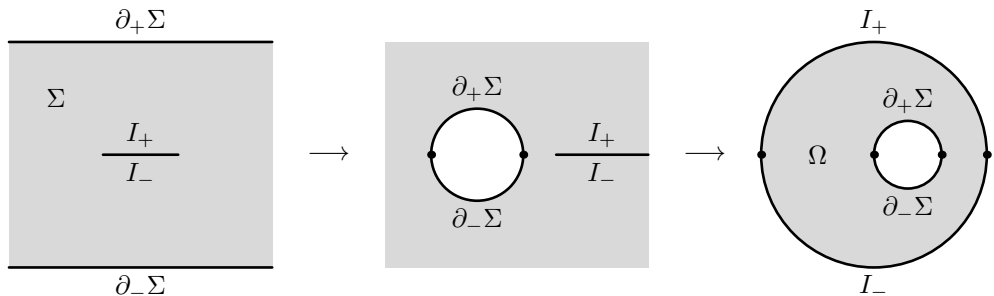}
\caption{A conformal transformation $\varkappa$ from the domain
$\Sigma$ defined in~\eqref{e:Sigma} onto
an annulus shaped domain $\Omega$ with smooth boundary,
depicted here as a composition of two transformations.
We mark the images of the components of $\partial\Sigma$.
If we put for simplicity $r:=\pi/2$, $I_0:=[0,\log 2]$,
then the first transformation is
$\zeta={e^z-1\over 2-e^z}$ and the second one
is $w={1\over 1-i\sqrt{\zeta}}$ where
we take the branch of the square root
which sends $\mathbb C\setminus [0,\infty)$ to
the upper half-plane.}
\label{f:annulus-shaped}
\end{figure}
%%%%%%%%%%%%%%%%%%%%%%%%%%%%%%%%%%%%%%%%%%%%%%%%%%%%%%%%%%%%%%%%%%%%%%%%%%%%%%%%

We henceforth consider the following domain in $\mathbb C$:
\begin{equation}
  \label{e:Sigma}
\Sigma:=\{x+iy\mid x\in\mathbb R,\ |y|<r\}\setminus I_0
\end{equation}
where $I_0\subset\mathbb R$ is an interval with $0<|I_0|\leq 1$ and $r\in(0,1)$.
The domain $\Sigma$ is unbounded and it does not have a smooth boundary
because of the slit $I_0$, however one can still define the harmonic measure
$\mu_t^\Sigma$ for each $t\in\Sigma$. 
One way to see this is by
taking a conformal transformation $\varkappa$ which maps $\Sigma$ onto an
annulus shaped domain $\Omega$ with smooth boundary, see Figure~\ref{f:annulus-shaped}
and~\cite[\S2.3]{Aleman}, and define $\mu_t^\Sigma$ by~\eqref{e:harmonic-conformal}.
However $\varkappa$ does not extend to a homeomorphism
$\partial\Sigma\to\partial\Omega$ since the images of sequences approaching the same point of $I_0$ from the top and from the bottom have different limits. To fix this issue,
we redefine $\partial\Sigma$ as
consisting of two lines
\begin{equation}
  \label{e:slit-boundary}
\partial_\pm \Sigma:=\{x\pm ir\mid x\in\mathbb R\}
\end{equation}
and two copies $I_\pm$ of the interval $I_0$ corresponding
to limits as $\Im z\to \pm 0$. This is in agreement with the Brownian
motion interpretation as it encodes in which direction $W_\tau$ crosses $I_0$.

The importance of harmonic measures in this paper is due to the following
%%%%%%%%%%%%%%%%%%%%%%%%%%%%%%%%%%%%%%%%%%%%%%%%%%%%%%%%%%%%%%%%%%%%%%%%%%%%%%%%
\begin{lemm}
  \label{l:harm-upper-log}
Assume that the function $F$ is holomorphic and bounded
on $\Sigma$ and extends continuously to $\partial\Sigma$.
Then for each $t\in\Sigma$ we have
\begin{equation}
  \label{e:harm-upper-log}
\log |F(t)|\leq \int_{\partial\Sigma} \log |F|\,d\mu_t^\Sigma.
\end{equation}
\end{lemm}
%%%%%%%%%%%%%%%%%%%%%%%%%%%%%%%%%%%%%%%%%%%%%%%%%%%%%%%%%%%%%%%%%%%%%%%%%%%%%%%%
\begin{proof}
The function $\log |F|$ is subharmonic and bounded above on $\Sigma$.
Then~\eqref{e:harm-upper-log} follows from Perron's construction
of solutions to the Dirichlet problem via subharmonic functions,
see for instance~\cite[\S19.7]{Conway}.
\end{proof}
%%%%%%%%%%%%%%%%%%%%%%%%%%%%%%%%%%%%%%%%%%%%%%%%%%%%%%%%%%%%%%%%%%%%%%%%%%%%%%%%
We now show several estimates on the harmonic measure
$\mu_t^\Sigma$ for the domain~\eqref{e:Sigma}.
Denote by $d(t,I_0)$ the distance from $t\in\mathbb R$ to the interval $I_0$.
In Lemmas~\ref{l:harm-upper-1}--\ref{l:harm-lower} below,
the precise dependence of the bounds on $I_0$ is irrelevant. However,
the dependence on $r$ is important.
%%%%%%%%%%%%%%%%%%%%%%%%%%%%%%%%%%%%%%%%%%%%%%%%%%%%%%%%%%%%%%%%%%%%%%%%%%%%%%%
\begin{lemm}
  \label{l:harm-upper-1}
Assume that $t\in\Sigma\cap\mathbb R$ satisfies $d(t,I_0)\geq {1\over 10}|I_0|$. Then 
\begin{equation}
  \label{e:harm-upper-1}
\Big\|{d\mu^\Sigma_t\over d\mu_L}\Big\|_{L^p(I_\pm)}\leq C_p\quad\text{for all }p\in [1,2)
\end{equation}
where $C_p>0$ depends only on $|I_0|$ and $p$.
\end{lemm}
%%%%%%%%%%%%%%%%%%%%%%%%%%%%%%%%%%%%%%%%%%%%%%%%%%%%%%%%%%%%%%%%%%%%%%%%%%%%%%%%
\begin{proof}
Without loss
of generality, we may assume that $I_0=[0,\ell]$ where $0<\ell\leq 1$.
We have $\Sigma\subset\widetilde\Sigma$ where
$$
\widetilde\Sigma:=\mathbb C\setminus I_0.
$$
Then (see for instance~\cite[Corollary 21.1.14]{Conway})
\begin{equation}
  \label{e:harmonic-monotone}
\mu^\Sigma_t|_{I_\pm}\ \leq\ \mu_t^{\widetilde\Sigma}|_{I_\pm}.
\end{equation}
We can also interpret~\eqref{e:harmonic-monotone} in stochastic terms:
every trajectory of
the Brownian motion starting at $t$ which hits
some $A\subset I_\pm$ before hitting $\partial\Sigma\setminus A$
also has the property that it hits $A$ before $\partial\widetilde\Sigma\setminus A$.

To compute $\mu_t^{\widetilde\Sigma}$ we use the conformal transformation
$$
z\mapsto w=\sqrt{{t-\ell\over t}\cdot {z\over \ell-z}}
$$
which maps $\widetilde\Sigma$ to the upper half-plane,
$I_\pm$ to $\pm [0,\infty)$, and $t$ to $i$. Using
the well-known formula for the harmonic measure of the upper half-plane,
we get
$$
\mu^{\widetilde\Sigma}_t|_{I_\pm}={|dw|\over \pi(1+w^2)}=
\sqrt{t(t-\ell)\over z(\ell-z)}\cdot {|dz|\over 2\pi |t-z|}.
$$
Since $|t-z|\geq {\ell \over 10}$ for all $z\in [0,\ell]$, it follows that
$$
{d\mu^\Sigma_t\over d\mu_L}(z)\leq {1\over \sqrt{z(\ell-z)}}\quad\text{for all }z\in I_\pm
$$
which implies~\eqref{e:harm-upper-1}.
\end{proof}
%%%%%%%%%%%%%%%%%%%%%%%%%%%%%%%%%%%%%%%%%%%%%%%%%%%%%%%%%%%%%%%%%%%%%%%%%%%%%%%%
%
%%%%%%%%%%%%%%%%%%%%%%%%%%%%%%%%%%%%%%%%%%%%%%%%%%%%%%%%%%%%%%%%%%%%%%%%%%%%%%%%
\begin{lemm}
  \label{l:harm-upper-2}
Assume that $t\in\Sigma\cap\mathbb R$ satisfies $d(t,I_0)\leq 1$. Then
$$
{d\mu_t^\Sigma\over d\mu_L}(x\pm ir)\leq {2\over r}e^{-d(x,I_0)},\quad
x\in\mathbb R.
$$
\end{lemm}
%%%%%%%%%%%%%%%%%%%%%%%%%%%%%%%%%%%%%%%%%%%%%%%%%%%%%%%%%%%%%%%%%%%%%%%%%%%%%%%%
\begin{proof}
We have $\Sigma\subset\widetilde\Sigma$ where
$$
\widetilde\Sigma:=\{(x+iy)\mid x\in\mathbb R,\ |y|<r\}.
$$
Therefore, similarly to~\eqref{e:harmonic-monotone} we have
$\mu_t^\Sigma|_{\partial_\pm\Sigma}\leq \mu_t^{\widetilde\Sigma}|_{\partial_\pm\Sigma}$.
The conformal transformation
$$
z\mapsto w=i\exp\Big({\pi (z-t)\over 2r}\Big)
$$
maps $\widetilde\Sigma$ to the upper half-plane,
$\partial_\pm\Sigma$ to $\mp [0,\infty)$,
and $t$ to $i$. Therefore
$$
\mu^{\widetilde\Sigma}_t|_{\partial_\pm \Sigma}={|dw|\over \pi(1+w^2)}=
{|dz|\over 4r\cosh\Big(\pi{\Re z-t\over 2r}\Big)}.
$$
It follows that
$$
\Big|{d\mu^\Sigma_t\over d\mu_L}(x\pm ir)\Big|\
\leq\ {1\over 2r}e^{-{\pi |x-t|\over 2r}}\
\leq\ {1\over 2r}e^{-|x-t|}\
\leq\ {2\over r}e^{-d(x,I_0)}.\qedhere
$$
\end{proof}
%%%%%%%%%%%%%%%%%%%%%%%%%%%%%%%%%%%%%%%%%%%%%%%%%%%%%%%%%%%%%%%%%%%%%%%%%%%%%%%%
%
%%%%%%%%%%%%%%%%%%%%%%%%%%%%%%%%%%%%%%%%%%%%%%%%%%%%%%%%%%%%%%%%%%%%%%%%%%%%%%%%
\begin{lemm}
  \label{l:harm-lower}
Assume that $t\in\Sigma\cap\mathbb R$ satisfies $d(t,I_0)\leq 1$.
Then
$$
\mu_t^\Sigma(I_\pm)\geq {|I_0|\over 8}e^{-2/r}.
$$
\end{lemm}

%%%%%%%%%%%%%%%%%%%%%%%%%%%%%%%%%%%%%%%%%%%%%%%%%%%%%%%%%%%%%%%%%%%%%%%%%%%%%%%%
\begin{proof}
Without loss of generality, we may assume that $I_0=[-\ell,0]$ where $0<\ell\leq 1$,
and $0<t\leq 1$.
We have $\Sigma\supset\widetilde\Sigma$ where
$$
\widetilde\Sigma=\{x+iy\mid x\in\mathbb R,\ |y|<r\}\setminus (-\infty,0].
$$
Similarly to~\eqref{e:harmonic-monotone} we have
$\mu_t^\Sigma(I_\pm)\geq \mu_t^{\widetilde\Sigma}(I_\pm)$.
The conformal transformation
$$
z\mapsto w=\sqrt{1-e^{\pi z/r}\over e^{\pi t/r}-1}
$$
maps $\widetilde\Sigma$ to the upper half-plane,
$t$ to $i$, and 
$$
I_\pm\
\mapsto\ \mp\bigg[0,
\sqrt{1-e^{-\pi\ell/r}\over e^{\pi t/r}-1}
\bigg]\ \supset\ \mp\Big[0, {\sqrt{\ell}\over 2}e^{-{\pi \over 2r}}\Big].
$$
It follows that
$$
\mu_t^\Sigma(I_\pm)\ \geq\
\mu_t^{\widetilde\Sigma}(I_\pm)
\ \geq\
{1\over \pi}\arctan\Big({\sqrt{\ell}\over 2}e^{-{\pi \over 2r}}\Big)\ \geq\
{\ell\over 8}e^{-2/r}.\qedhere
$$
\end{proof}
%%%%%%%%%%%%%%%%%%%%%%%%%%%%%%%%%%%%%%%%%%%%%%%%%%%%%%%%%%%%%%%%%%%%%%%%%%%%%%%%

%%%%%%%%%%%%%%%%%%%%%%%%%%%%%%%%%%%%%%%%%%%%%%%%%%%%%%%%%%%%%%%%%%%%%%%%%%%%%%%%
%%%%%%%%%%%%%%%%%%%%%%%%%%%%%%%%%%%%%%%%%%%%%%%%%%%%%%%%%%%%%%%%%%%%%%%%%%%%%%%%
\section{General fractal uncertainty principle}
  \label{s:general-fup}

In this section, we prove Theorem~\ref{t:general-fup}.
We establish the components
of the argument in~\S\S\ref{s:adapted-multiplier},\ref{s:harmonic-fourier-bound}
and combine them in~\S\ref{s:step} to obtain a unique continuation
estimate for functions with Fourier supports in regular sets.
In~\S\ref{s:iteration-argument}, we iterate this
estimate to finish the proof.

%%%%%%%%%%%%%%%%%%%%%%%%%%%%%%%%%%%%%%%%%%%%%%%%%%%%%%%%%%%%%%%%%%%%%%%%%%%%%%%%
\subsection{An adapted multiplier}
  \label{s:adapted-multiplier}

We first construct a compactly supported function whose
Fourier transform decays much faster than $\exp(-|\xi|/\log |\xi|)$
as $|\xi|\to\infty$ on a $\delta$-regular set. We henceforth denote
\begin{equation}
  \label{e:theta-def}
\theta(\xi):=\log(10+|\xi|)^{-{1+\delta\over 2}}.
\end{equation}
The function $\psi$ constructed in the lemma below is used as a convolution kernel
in the proof of Lemma~\ref{l:smaller-step}.
We remark that $\alpha_1$ is a finite but large parameter,
and it is important that the constants in the estimates do not depend on $\alpha_1$.
%%%%%%%%%%%%%%%%%%%%%%%%%%%%%%%%%%%%%%%%%%%%%%%%%%%%%%%%%%%%%%%%%%%%%%%%%%%%%%%%
\begin{lemm}
  \label{l:special-multiplier}
Assume that $Y\subset [-\alpha_1,\alpha_1]$ is a $\delta$-regular set with constant $C_R$
on scales~$2$ to~$\alpha_1$, and $\delta\in (0,1)$. Fix $c_1>0$.
Then there exist a constant $c_2>0$ depending only on $\delta,C_R,c_1$
and a function $\psi\in L^2(\mathbb R)$ such that
\begin{gather}
  \label{e:specmul-1}
\supp\psi\subset \Big[-{c_1\over 10},{c_1\over 10}\Big],\\
  \label{e:specmul-2}
\|\widehat\psi\|_{L^2([-1,1])}\geq c_2,\\
  \label{e:specmul-3}
|\widehat\psi(\xi)|\leq \exp(-c_2\langle\xi\rangle^{1/2})\quad\text{for all }\xi\in\mathbb R,\\
  \label{e:specmul-4}
|\widehat\psi(\xi)|\leq \exp\big(-c_2\theta(\xi)|\xi|\big)\quad\text{for all }\xi\in Y.
\end{gather}
\end{lemm}
%%%%%%%%%%%%%%%%%%%%%%%%%%%%%%%%%%%%%%%%%%%%%%%%%%%%%%%%%%%%%%%%%%%%%%%%%%%%%%%%
\Remarks 1. It is essential that condition~\eqref{e:specmul-4}
be imposed only on $Y$. Indeed,
it is a standard fact in harmonic analysis (see for instance~\cite[\S1.5.4]{Havin-Book})
that every compactly supported $\psi\in L^2(\mathbb R)$ with $\psi\not\equiv 0$ satisfies
\begin{equation}
  \label{e:weight-condition}
\int_{\mathbb R}{\log |\widehat\psi(\xi)|\over 1+|\xi|^2}\,d\xi > -\infty
\end{equation}
which would contradict~\eqref{e:specmul-4} if $Y$ were replaced by $\mathbb R$.

\noindent 2. In~\eqref{e:specmul-3} one could replace $\exp(-\langle\xi\rangle^{1/2})$ by
any weight satisfying~\eqref{e:bmmt-thm-1}, \eqref{e:bmmt-thm-2}
and decaying as $|\xi|\to\infty$ faster than any negative power of $|\xi|$.
Also, the proof below works with $1+\delta\over 2$ replaced by 1, though
in that case~\eqref{e:specmul-4} would not suffice for our application.

\noindent 3. Recently Jin--Zhang~\cite{JinZhang} have shown that
the Hilbert transform of
the logarithm of the weight $\omega$ constructed in the proof below has uniformly bounded Lipschitz constant. Then Lemma~\ref{l:special-multiplier}
can be proved using a weaker (and considerably easier to prove) version
of the Beurling--Malliavin Theorem~\cite[Theorem~1]{Havin}.
%%%%%%%%%%%%%%%%%%%%%%%%%%%%%%%%%%%%%%%%%%%%%%%%%%%%%%%%%%%%%%%%%%%%%%%%%%%%%%%%
\begin{proof}
We will use Lemma~\ref{l:bmmt}.
For this we construct a weight
adapted to the set $Y$.
Define $n_1\in\mathbb N$ by the inequality
$2^{n_1}\leq \alpha_1<2^{n_1+1}$.
For every $n\in\mathbb N$, $n\leq n_1$, put
$$
A_n:=[-2^{n+1},-2^n]\sqcup [2^n,2^{n+1}],\quad
\rho_n := n^{-{1+\delta\over 2}}\cdot 2^n\geq 2.
$$
Using Lemma~\ref{l:regular-cover}, construct a nonoverlapping collection
$\mathcal J_n$ of $N_n$ intervals of size
$\rho_n$ each such that all elements of $\mathcal J_n$
intersect $A_n$ and
$$
Y\cap A_n\subset\bigcup_{J\in\mathcal J_n}J,\quad
N_n\leq 24C_R^2\cdot \Big({2^n\over\rho_n}\Big)^\delta.
$$
Fix a cutoff function
$$
\chi(\xi)\in C^1(\mathbb R;[0,1]),\quad
\sup|\partial_\xi\chi|\leq 10,\quad
\supp\chi\subset [-1,1],\quad
\chi=1\quad\text{on }\Big[-{1\over 2},{1\over 2}\Big].
$$
For an interval $J$ with center $\xi_J$, define
the function $\chi_J\in C^1(\mathbb R)$ by
$$
\chi_J(\xi)=|J|\cdot\chi\Big({\xi-\xi_J\over |J|}\Big),
$$
so that
$$
\begin{gathered}
0\leq \chi_J\leq |J|,\quad
\sup|\partial_\xi\chi_J|\leq 10,\\
\supp\chi_J\subset \widetilde J:=\xi_J+\big[-|J|,|J|\big],\quad
\chi_J=|J|\quad\text{on }J.
\end{gathered}
$$
Now, define the weight $\omega\in C^1(\mathbb R;(0,1])$ by (see Figure~\ref{f:weight})
%%%%%%%%%%%%%%%%%%%%%%%%%%%%%%%%%%%%%%%%%%%%%%%%%%%%%%%%%%%%%%%%%%%%%%%%%%%%%%%%
\begin{figure}
\includegraphics{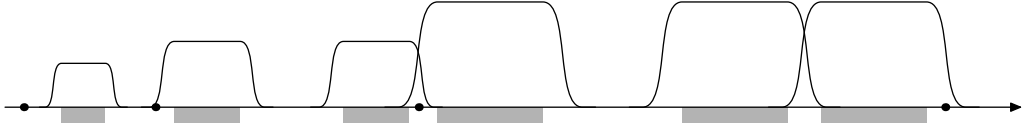}
\caption{The functions $\chi_J$ featured in~\eqref{e:omega-weight},
where the intervals $J$ are shaded.
The dots mark powers of 2.}
\label{f:weight}
\end{figure}
%%%%%%%%%%%%%%%%%%%%%%%%%%%%%%%%%%%%%%%%%%%%%%%%%%%%%%%%%%%%%%%%%%%%%%%%%%%%%%%%
\begin{equation}
  \label{e:omega-weight}
\omega(\xi):=\exp(-2\langle\xi\rangle^{1/2})\cdot\prod_{n=1}^{n_1}\prod_{J\in\mathcal J_n}
\exp(-10\chi_J).
\end{equation}
For each $\xi\in Y$, $|\xi|\geq 2$,
there exists $n\in [1,n_1]$ and $J\in\mathcal J_n$ such that $\xi\in J$. 
Also, $\exp(-2\langle\xi\rangle^{1/2})\leq \exp(-\theta(\xi)|\xi|)$ for $|\xi|\leq 2$. Therefore
\begin{gather}
  \label{e:wt-1}
\omega(\xi)\leq \exp(-\langle\xi\rangle^{1/2})\quad\text{for all }\xi\in\mathbb R,\\
  \label{e:wt-2}
\omega(\xi)\leq \exp\big(-\theta(\xi)|\xi|\big)\quad\text{for all }\xi\in Y.
\end{gather}
Since each $\xi$ lies in at most 500 intervals in $\bigcup_{n=1}^{n_1}
\bigcup_{J\in\mathcal J_n}\widetilde J$,
we have
$$
\sup|\partial_\xi\log\omega|\leq 10^5.
$$
Next,
$$
\begin{aligned}
\int_{\mathbb R}{|\log\omega(\xi)|\over 1+\xi^2}\,d\xi&\leq 100+10^5\sum_{n=1}^{n_1}N_n\cdot 2^{-2n}\rho_n^2\\
&\leq 10^5+10^7C_R^2\sum_{n=1}^{\infty}\Big({2^n\over\rho_n}\Big)^{\delta-2}=:C_0
\end{aligned}
$$
where $C_0$ depends only on $\delta,C_R$. Here we use the formula
for $\rho_n$ and the inequality $(1+\delta)(1-\delta/2)>1$
valid for all $\delta\in (0,1)$.

We have verified that the weight $\omega$ satisfies~\eqref{e:bmmt-1} and~\eqref{e:bmmt-2}.
Applying Lemma~\ref{l:bmmt} with $c_0:=c_1/10$, we construct
$\psi\in L^2(\mathbb R)$ satisfying~\eqref{e:specmul-1}
and
$$
|\widehat\psi|\leq \omega^{c_2},\quad
\|\widehat\psi\|_{L^2(-1,1)}\geq c_2,
$$
where the constant $c_2$ depends only on $\delta,C_R,c_1$.
By~\eqref{e:wt-1} and~\eqref{e:wt-2}, $\psi$ satisfies \eqref{e:specmul-3} and~\eqref{e:specmul-4}.
\end{proof}
%%%%%%%%%%%%%%%%%%%%%%%%%%%%%%%%%%%%%%%%%%%%%%%%%%%%%%%%%%%%%%%%%%%%%%%%%%%%%%%%

%%%%%%%%%%%%%%%%%%%%%%%%%%%%%%%%%%%%%%%%%%%%%%%%%%%%%%%%%%%%%%%%%%%%%%%%%%%%%%%%
\subsection{A bound on functions with compact Fourier support}
  \label{s:harmonic-fourier-bound}

We next use the harmonic measure estimates from~\S\ref{s:harmonic-measures}
to obtain the following quantitative unique continuation estimate
which is used in the proof of Lemma~\ref{l:smaller-step} below.
%%%%%%%%%%%%%%%%%%%%%%%%%%%%%%%%%%%%%%%%%%%%%%%%%%%%%%%%%%%%%%%%%%%%%%%%%%%%%%%%
\begin{lemm}
  \label{l:harmest-main}
Assume that $\mathcal I$ is a nonoverlapping collection of intervals
of size 1 each, and for each $I\in\mathcal I$ we choose a subinterval
$I''\subset I$ with $|I''|=c_0>0$ independent of~$I$.
Then there exists a constant $C$ depending only on $c_0$ such that
for all $r\in (0,1)$, $0<\kappa\leq e^{-C/r}$,
and $f\in L^2(\mathbb R)$ with $\hat f$ compactly supported, we have
\begin{equation}
  \label{e:harmest-main}
\sum_{I\in\mathcal I}\|f\|_{L^2(I)}^2
\leq {C\over r}\bigg(\sum_{I\in\mathcal I}\|f\|_{L^2(I'')}^2\bigg)^\kappa
\cdot \|e^{2\pi r|\xi|}\hat f(\xi)\|_{L^2(\mathbb R)}^{2(1-\kappa)}.
\end{equation}
\end{lemm}
%%%%%%%%%%%%%%%%%%%%%%%%%%%%%%%%%%%%%%%%%%%%%%%%%%%%%%%%%%%%%%%%%%%%%%%%%%%%%%%%
\Remark
The bound~\eqref{e:intro-interpol} in the introduction follows
from~\eqref{e:harmest-main}. To see this, take large $K$ to be chosen later
and decompose $f=f_1+f_2$ where $\supp\hat f_1\subset [-K,K]$,
$\supp\hat f_2\subset \mathbb R\setminus (-K,K)$.
Put $r:={1\over 10}\theta(K)$ and apply~\eqref{e:harmest-main} to $f_1$
(with $I'$ taking the role of $I''$):
\begin{equation}
  \label{e:fort-1}
\|f_1\|_{L^2(\mathbb R)}^2\leq {C\over \theta(K)}\|f_1\|_{L^2(U')}^{2\kappa}
\cdot \|f\|_{L^2(\mathbb R)}^{2(1-\kappa)}
\end{equation}
where we use that $\|e^{2\pi r|\xi|}\hat f_1(\xi)\|_{L^2}\leq \|\exp(\theta(\xi)|\xi|)\hat f(\xi)\|_{L^2}
\leq C_1\|f\|_{L^2}$.
Moreover
\begin{equation}
  \label{e:fort-2}
\|f_2\|_{L^2}\leq e^{-\theta(K)K}\|\exp(\theta(\xi)|\xi|)\hat f(\xi)\|_{L^2}
\leq C_1e^{-\theta(K)K}\|f\|_{L^2}.
\end{equation}
We have $\|f_1\|_{L^2(U')}^{2\kappa}\leq C(\|f\|_{L^2(U')}^{2\kappa}+\|f_2\|_{L^2(\mathbb R)}^{2\kappa})$.
Combining~\eqref{e:fort-1} with~\eqref{e:fort-2}, we get
\begin{equation}
  \label{e:fort-3}
\|f\|_{L^2}^2=\|f_1\|_{L^2}^2+\|f_2\|_{L^2}^2\leq
{C\over \theta(K)}\big(
\|f\|_{L^2(U')}^{2\kappa}\cdot \|f\|_{L^2}^{2(1-\kappa)}
+e^{-2\theta(K)\kappa K}\|f\|_{L^2}^2\big)
\end{equation}
where the constant $C$ depends only on $c_1,C_1$. Since $\delta<1$ we have
$e^{-2\theta(K)\kappa K}/\theta(K)\to 0$ as $K\to\infty$.
We then fix $K$ large enough depending on $\delta,c_1,C_1$
to remove the last term on the right-hand side of~\eqref{e:fort-3},
giving~\eqref{e:intro-interpol}. The proof of the unique continuation
bound in~\S\ref{s:step} is inspired by the above argument.

%%%%%%%%%%%%%%%%%%%%%%%%%%%%%%%%%%%%%%%%%%%%%%%%%%%%%%%%%%%%%%%%%%%%%%%%%%%%%%%%
\begin{proof}[Proof of Lemma~\ref{l:harmest-main}]
Since $\hat f$ is compactly supported, $f$ has a
holomorphic continuation $F$ given by~\eqref{e:inverse-fourier}:
$$
F(z)=\int_{\mathbb R} e^{2\pi iz\xi}\hat f(\xi)\,d\xi,\quad
z\in\mathbb C;\quad
f=F|_{\mathbb R}.
$$
The function $F(z)$ is bounded on $\{|\Im z|\leq r\}$
and 
\begin{equation}
  \label{e:side-bound}
\int_{\mathbb R} |F(x\pm ir)|^2 \,dx
=\int_{\mathbb R} |e^{\mp 2\pi r\xi}\hat f(\xi)|^2\,d\xi
\leq \|e^{2\pi r|\xi|}\hat f(\xi)\|_{L^2}^2.
\end{equation}
For each $I\in \mathcal I$, let $I_0\Subset I''$ be the interval
with the same center as $I''$ and $|I_0|={1\over 2}c_0$. Define the slit domain
(see Figure~\ref{f:special-slit})
$$
\Sigma_I:=\{x+iy\mid x\in\mathbb R,\ |y|<r\}\setminus I_0.
$$
For each $t\in I\setminus I''\subset\Sigma_I$, let $\mu_t=\mu_t^{\Sigma_I}$ be the harmonic measure of $\Sigma_I$ on
$$
\partial\Sigma_I=I_0\sqcup \partial_-\Sigma_I\sqcup \partial_+\Sigma_I,\quad
\partial_\pm\Sigma_I=\{x\pm ir\mid x\in\mathbb R\}
$$
centered at $t$. Here we put together the top and bottom copies
$I_\pm$ of $I_0$ (see the paragraph following~\eqref{e:slit-boundary}),
that is for $A\subset I_0$ we have
$\mu_t(A)=\mu_t(A\cap I_+)+\mu_t(A\cap I_-)$.

%%%%%%%%%%%%%%%%%%%%%%%%%%%%%%%%%%%%%%%%%%%%%%%%%%%%%%%%%%%%%%%%%%%%%%%%%%%%%%%%
\begin{figure}
\includegraphics{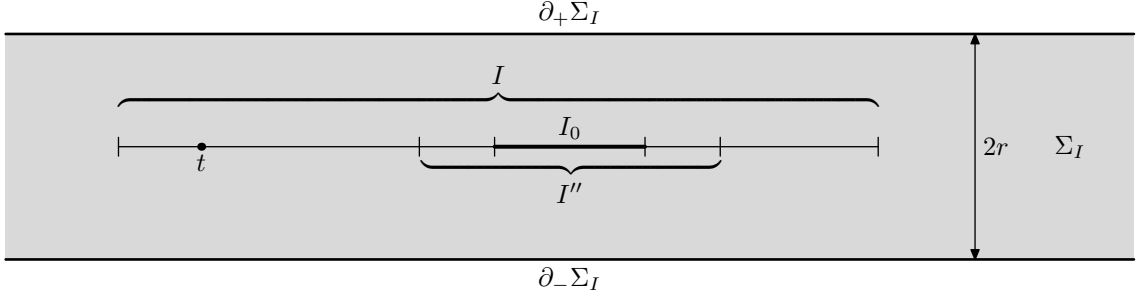}
\caption{The slit domain $\Sigma_I$ with the intervals $I_0\subset I''\subset I$.}
\label{f:special-slit}
\end{figure}
%%%%%%%%%%%%%%%%%%%%%%%%%%%%%%%%%%%%%%%%%%%%%%%%%%%%%%%%%%%%%%%%%%%%%%%%%%%%%%%%
By Lemma~\ref{l:harm-lower}, we have
$$
\kappa_I:=\mu_t(I_0)\ \geq\ {c_0\over 8}e^{-2/r}\ \geq\ e^{-C/r}\ \geq\ \kappa
$$
where $C$ denotes a constant depending only
on $c_0$ (whose value might differ in different parts of the proof).
By Lemma~\ref{l:harm-upper-log} we estimate
$$
\begin{gathered}
2\log |f(t)|\leq \int_{\partial\Sigma_I} 2\log |F(z)|\,d\mu_t(z) \\
=4\kappa_I\cdot {1\over\kappa_I}\int_{I_0} {\log |f(x)|\over 2}\,d\mu_t(x)
+(1-\kappa_I)\cdot {1\over 1-\kappa_I}\int_{\partial_-\Sigma_I\sqcup\partial_+\Sigma_I}
2\log |F(z)|\,d\mu_t(z).
\end{gathered}
$$
Since the exponential function is convex
and ${1\over\kappa_I}\mu_t|_{I_0}$, ${1\over 1-\kappa_I}\mu_t|_{\partial_-\Sigma_I\sqcup
\partial_+\Sigma_I}$ are probability measures,
we obtain
$$
|f(t)|^2\leq\bigg({1\over\kappa_I}\int_{I_0} |f(x)|^{1/2}\,d\mu_t(x)\bigg)^{4\kappa_I}\cdot
\bigg(
{1\over 1-\kappa_I}
\int_{\partial_-\Sigma_I\sqcup\partial_+\Sigma_I}|F(z)|^2\,d\mu_t(z)
\bigg)^{1-\kappa_I}.
$$
Since $\kappa\leq\kappa_I<1$
and $\lambda^{-\lambda}\leq \exp(1/e)$ for all $\lambda >0$
it follows that
\begin{equation}
  \label{e:convex-exp-used}
\begin{aligned}
|f(t)|^2\leq\,& 10\bigg(\int_{I_0} |f(x)|^{1/2}\,d\mu_t(x)\bigg)^{4\kappa}\\&\cdot \bigg(
\bigg(\int_{I_0} |f(x)|^{1/2}\,d\mu_t(x)\bigg)^4+\int_{\partial_-\Sigma_I\sqcup\partial_+\Sigma_I}|F(z)|^2\,d\mu_t(z)
\bigg)^{1-\kappa}.
\end{aligned}
\end{equation}
Recall that $t\in I\setminus I''$.
By Lemma~\ref{l:harm-upper-1} with $p=4/3$ and H\"older's inequality we have
\begin{equation}
  \label{e:hu-1}
\bigg(\int_{I_0} |f(x)|^{1/2}\,d\mu_t(x)\bigg)^4\leq C\|f\|_{L^2(I_0)}^2
\end{equation}
and by Lemma~\ref{l:harm-upper-2}
\begin{equation}
  \label{e:hu-2}
\int_{\partial_-\Sigma_I\sqcup\partial_+\Sigma_I}|F(z)|^2\,d\mu_t(z)
\leq {C\over r}\int_{\Im z\in \{\pm r\}}e^{-d(\Re z,I)}|F(z)|^2 dz.
\end{equation}
Combining~\eqref{e:convex-exp-used}--\eqref{e:hu-2} we get
$$
|f(t)|^2\leq {C\over r}\|f\|_{L^2(I_0)}^{2\kappa}
\cdot\bigg(\int_{\Im z\in\{0,\pm r\}}
e^{-d(\Re z,I)}|F(z)|^2\,dz
 \bigg)^{1-\kappa}.
$$
Integrating in $t\in I\setminus I''$ and using H\"older's inequality~\eqref{e:holder-special},
we estimate
$$
\begin{aligned}
\sum_{I\in\mathcal I}\|f\|_{L^2(I\setminus I'')}^2&\leq 
{C\over r}\bigg(\sum_{I\in \mathcal I}\|f\|_{L^2(I_0)}^2\bigg)^\kappa
\cdot\bigg(\sum_{I\in\mathcal I}
\int_{\Im z\in\{0,\pm r\}}
e^{-d(\Re z,I)}|F(z)|^2\,dz
\bigg)^{1-\kappa}
\\&\leq
{C\over r}\bigg(\sum_{I\in \mathcal I}\|f\|_{L^2(I'')}^2\bigg)^\kappa
\cdot \bigg(\int_{\Im z\in \{0,\pm r\}}|F(z)|^2\,dz\bigg)^{1-\kappa}.
\end{aligned}
$$
Combining this with~\eqref{e:side-bound}
and the bound
$$
\sum_{I\in\mathcal I}\|f\|_{L^2(I'')}^2\leq
\bigg(\sum_{I\in \mathcal I}\|f\|_{L^2(I'')}^2\bigg)^\kappa
\cdot \|e^{2\pi r|\xi|}\hat f(\xi)\|_{L^2}^{2(1-\kappa)}
$$
we obtain~\eqref{e:harmest-main}.
\end{proof}
%%%%%%%%%%%%%%%%%%%%%%%%%%%%%%%%%%%%%%%%%%%%%%%%%%%%%%%%%%%%%%%%%%%%%%%%%%%%%%%%

%%%%%%%%%%%%%%%%%%%%%%%%%%%%%%%%%%%%%%%%%%%%%%%%%%%%%%%%%%%%%%%%%%%%%%%%%%%%%%%%
\subsection{The iterative step}
  \label{s:step}

The key component of the proof of Theorem~\ref{t:general-fup}
is the following unique continuation
property for functions with Fourier support in a $\delta$-regular set:
%%%%%%%%%%%%%%%%%%%%%%%%%%%%%%%%%%%%%%%%%%%%%%%%%%%%%%%%%%%%%%%%%%%%%%%%%%%%%%%%
\begin{prop}
  \label{l:step}
Assume that $Y\subset [-\alpha_1,\alpha_1]$ is $\delta$-regular
with constant $C_R$ on scales $1$ to $\alpha_1$, and $\delta\in (0,1)$.
Take
\begin{equation}
  \label{e:cal-I}
\mathcal I:=\big\{[j,j+1]\mid j\in\mathbb Z\big\}
\end{equation}
and assume that for each $I\in\mathcal I$ we are given
a subinterval $I'\subset I$ with $|I'|=c_1 >0$ independent of $I$.
(See Figure~\ref{f:step}.)
Define
$$
U':=\bigcup_{I\in\mathcal I}I'.
$$
Then there exists $c_3>0$ depending only on $\delta,C_R,c_1$
such that for all $f\in L^2(\mathbb R)$ with $\supp \hat f\subset Y$, we have
\begin{equation}
  \label{e:step}
\|f\|_{L^2(U')}\geq c_3 \|f\|_{L^2(\mathbb R)}.
\end{equation}
\end{prop}
%%%%%%%%%%%%%%%%%%%%%%%%%%%%%%%%%%%%%%%%%%%%%%%%%%%%%%%%%%%%%%%%%%%%%%%%%%%%%%%%
\Remark It is important that $U'$ be the union of infinitely many intervals,
rather than a single interval. Indeed, the following estimate is false:
$$
f\in L^2(\mathbb R),\quad
\supp \hat f\subset [-1,1]\quad\Longrightarrow\quad
\|f\|_{L^2(-1,1)}\geq c\|f\|_{L^2(-2,2)}
$$
as can be seen by taking
$f(x)=x^N\chi(x)$, where
$\chi$ is a Schwartz function with $\supp\hat\chi\subset [-1,1]$,
and letting $N\to\infty$.
%%%%%%%%%%%%%%%%%%%%%%%%%%%%%%%%%%%%%%%%%%%%%%%%%%%%%%%%%%%%%%%%%%%%%%%%%%%%%%%%
\begin{figure}
\includegraphics{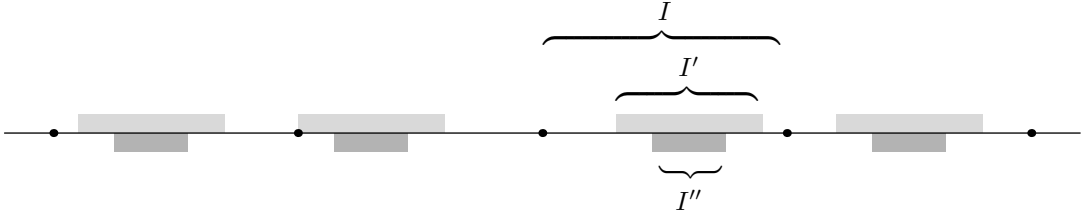}
\caption{The sets $U'$ (light shaded) and $U''$ (dark shaded) used in Proposition~\ref{l:step}
and Lemma~\ref{l:smaller-step}, with endpoints of the intervals $I\in\mathcal I$ denoted
by dots.}
\label{f:step}
\end{figure}
%%%%%%%%%%%%%%%%%%%%%%%%%%%%%%%%%%%%%%%%%%%%%%%%%%%%%%%%%%%%%%%%%%%%%%%%%%%%%%%%

Henceforth in this section $C$ denotes a constant which
only depends on $\delta,C_R,c_1$ (whose value may differ in different places).
Recall the definition~\eqref{e:theta-def} of $\theta(\xi)$.
For $f\in L^2$, denote by $\|f\|_{H^{-10}}$ its Sobolev norm
defined in~\eqref{e:sobolev-space};
it will be useful for summing over different
phase shifts of $f$ in~\eqref{e:ell-sum} below.

The main ingredient
of the proof of Proposition~\ref{l:step} is the following lemma,
which combines the results of~\S\S\ref{s:adapted-multiplier}--\ref{s:harmonic-fourier-bound}.
It is proved by splitting $f$ into two pieces, one which lives
on frequencies $\leq K$ and the other, on frequencies $\geq K$.
%%%%%%%%%%%%%%%%%%%%%%%%%%%%%%%%%%%%%%%%%%%%%%%%%%%%%%%%%%%%%%%%%%%%%%%%%%%%%%%%
\begin{lemm}
  \label{l:smaller-step}
Assume that $Z\subset [-\alpha_1,\alpha_1]$ is $\delta$-regular
with constant $C_R$ on scales $1$ to $\alpha_1\geq 2$, and $\delta\in (0,1)$.
Let $\{I'\}_{I\in\mathcal I}$ be as in Proposition~\ref{l:step}.
Then we have
for all $f\in L^2(\mathbb R)$ with $\supp\hat f\subset Z$,
all $K>10$, and $\kappa:=\exp(-C/\theta(K))$,
$$
\|\hat f\|_{L^2(-1,1)}^2\leq CK^{21}
\big(\|\indic_{U'}f\|_{H^{-10}}^2
+\exp\big(-C^{-1}\theta(K)K\big)\|f\|_{H^{-10}}^2\big)^{\kappa}
\cdot \|f\|_{H^{-10}}^{2(1-\kappa)}.
$$
\end{lemm}
%%%%%%%%%%%%%%%%%%%%%%%%%%%%%%%%%%%%%%%%%%%%%%%%%%%%%%%%%%%%%%%%%%%%%%%%%%%%%%%%
\begin{proof}
For each $I\in \mathcal I$, let $I''\Subset I'$ be the interval
with the same center as $I'$ and $|I''|={1\over 2}c_1$. Denote
$$
U'':=\bigcup_{I\in\mathcal I}I''.
$$
Let $\psi$ be the function constructed in Lemma~\ref{l:special-multiplier}
for $Y$ replaced by
$$
Z(2):=Z+[-2,2].
$$
Here $Z(2)\subset [-(\alpha_1+2),\alpha_1+2]$ is a $\delta$-regular set with constant $100C_R$ on scales 2 to $\alpha_1+2$
by Lemmas~\ref{l:regular-fatten} and~\ref{l:regular-expand-top}.
By \eqref{e:specmul-2}--\eqref{e:specmul-4} we have
for some $c_2\in(0,1)$ depending only on $\delta,C_R,c_1$
\begin{gather}
  \label{e:gft-0}
\|\widehat\psi\|_{L^2(-1,1)}\geq c_2,\\
  \label{e:gft-1.5}
|\widehat\psi(\xi)|\leq \exp(-c_2\langle\xi\rangle^{1/2})\quad\text{for all }\xi\in\mathbb R,
\\
  \label{e:gft-1}
|\widehat\psi(\xi)|\leq 
\exp\big(-c_2\theta(\xi)|\xi|\big)\quad\text{for all }\xi\in Z(2).
\end{gather}
Take arbitrary $\eta\in [-2,2]$ and let $f_\eta(x):=e^{2\pi i \eta x}f(x)$, so that
$\hat f_\eta(\xi)=\hat f(\xi-\eta)$. 
The freedom of choice in $\eta$ will be useful in~\eqref{e:eta-finally-used} below,
for simplicity the reader can consider the case $\eta=0$. Put
$$
g_\eta:=f_\eta*\psi\in L^2(\mathbb R).
$$
By the support condition~\eqref{e:specmul-1},
\begin{equation}
  \label{e:g-still-there}
g_\eta=(\indic_{U'}f_\eta)*\psi\quad\text{on }U''.
\end{equation}
By~\eqref{e:convolution-fourier}, \eqref{e:gft-1},
 and since $\supp \hat f_\eta\subset Z+\eta\subset Z(2)$, we have
\begin{equation}
  \label{e:gft}
|\hat g_\eta(\xi)|\leq \exp\big(-c_2\theta(\xi)|\xi|\big)\cdot
|\hat f_\eta(\xi)|\quad\text{for all }\xi\in\mathbb R.
\end{equation}
Put
$$
r:={c_2\over 10}\theta(K)\in (0,1).
$$
Since $\theta(\xi)$ is decreasing for $\xi\geq 0$, we have
\begin{equation}
  \label{e:nota}
\sup_{|\xi|\leq K}e^{2\pi r|\xi|}\exp\big(-c_2\theta(\xi)|\xi|\big)\leq 1.
\end{equation}
We now decompose $g_\eta$ into low and high frequencies:
$$
g_\eta=g_1+g_2,\quad
g_1,g_2\in L^2,\quad
\supp \hat g_1\subset \{|\xi|\leq K\},\quad
\supp \hat g_2\subset \{|\xi|\geq K\}.
$$
Then by~\eqref{e:gft} and~\eqref{e:nota}
\begin{gather}
  \label{e:g1-estimate}
\|e^{2\pi r|\xi|}\hat g_1(\xi)\|_{L^2}\leq CK^{10}\|f\|_{H^{-10}},\\
  \label{e:g2-estimate}
\|g_2\|_{L^2}\leq C\exp\big(-C^{-1}\theta(K) K\big)\|f\|_{H^{-10}}.
\end{gather}
Applying Lemma~\ref{l:harmest-main} to the function $g_1$
and using~\eqref{e:g1-estimate}, we get
\begin{equation}
  \label{e:harmonic-used}
\|g_1\|_{L^2}^2\leq
{CK^{20}\over r}
\|g_1\|_{L^2(U'')}^{2\kappa}\cdot
\|f\|_{H^{-10}}^{2(1-\kappa)},\quad
\kappa:=e^{-C/r}.
\end{equation}
By~\eqref{e:g-still-there}, \eqref{e:convolution-fourier}, and~\eqref{e:gft-1.5}
$$
\begin{aligned}
\|g_1\|_{L^2(U'')}&\ \leq\ \|g_\eta\|_{L^2(U'')}+\|g_2\|_{L^2}
\ \leq\ \|(\indic_{U'}f_\eta)*\psi\|_{L^2}+\|g_2\|_{L^2}
\\
&\ \leq\ C\|\indic_{U'}f\|_{H^{-10}}
+\|g_2\|_{L^2}.
\end{aligned}
$$
Then by~\eqref{e:g2-estimate} and~\eqref{e:harmonic-used}
and since $r^{-1}\leq CK$,
we have for all $\eta\in [-2,2]$
$$
\begin{aligned}
\|g_\eta\|_{L^2}^2&=\|g_1\|_{L^2}^2+\|g_2\|_{L^2}^2
\\&\leq
CK^{21}
\big(\|\indic_{U'}f\|_{H^{-10}}^2
+\exp\big(-C^{-1}\theta(K)K\big)\|f\|_{H^{-10}}^2\big)^{\kappa}
\cdot \|f\|_{H^{-10}}^{2(1-\kappa)}.
\end{aligned}
$$
It remains to use the following corollary of~\eqref{e:gft-0}:
\begin{equation}
  \label{e:eta-finally-used}
\begin{aligned}
\|\hat f\|_{L^2(-1,1)}^2&\leq 
c_2^{-2}\int_{[-1,1]^2}|\hat f(\zeta)\widehat\psi(\xi)|^2\,d\xi d\zeta\\
&\leq c_2^{-2}\int_{-2}^2 \int_{\mathbb R}|\hat f(\xi-\eta)\widehat\psi(\xi)|^2\,d\xi d\eta\\
&=c_2^{-2}\int_{-2}^2 \|g_\eta\|_{L^2}^2\,d\eta
\end{aligned}
\end{equation}
where $\hat g_\eta(\xi)=\hat f(\xi-\eta)\widehat \psi(\xi)$ by~\eqref{e:convolution-fourier}.
\end{proof}
%%%%%%%%%%%%%%%%%%%%%%%%%%%%%%%%%%%%%%%%%%%%%%%%%%%%%%%%%%%%%%%%%%%%%%%%%%%%%%%%
Armed with Lemma~\ref{l:smaller-step}, we now give
%%%%%%%%%%%%%%%%%%%%%%%%%%%%%%%%%%%%%%%%%%%%%%%%%%%%%%%%%%%%%%%%%%%%%%%%%%%%%%%%
\begin{proof}[Proof of Proposition~\ref{l:step}]
Take $\ell\in \mathbb Z$ such that $|\ell|\leq \alpha_1$. 
By Lemmas~\ref{l:regular-scale} and~\ref{l:regular-expand-top}
the set $Y+\ell\subset [-2\alpha_1,2\alpha_1]$ is
$\delta$-regular with constant $4C_R$ on scales 1 to $2\alpha_1$.
Put
$$
f_\ell(x):=e^{2\pi i\ell x}f(x),
$$
then
$\hat f_\ell(\xi)=\hat f(\xi-\ell)$ and thus $\supp\hat f_\ell\subset Y+\ell$. By Lemma~\ref{l:smaller-step} applied to $f_\ell$ and
$Z:=Y+\ell$, we have for all $K>10$
and $\kappa:=\exp(-C/\theta(K))$
$$
\|\hat f_\ell\|_{L^2(-1,1)}^2\leq CK^{21}
\big(\|\indic_{U'}f_\ell\|_{H^{-10}}^2
+\exp\big(-C^{-1}\theta(K)K\big)\|f_\ell\|_{H^{-10}}^2\big)^{\kappa}
\cdot \|f_\ell\|_{H^{-10}}^{2(1-\kappa)}.
$$
Since $\supp\hat f\subset Y\subset [-\alpha_1,\alpha_1]$, we obtain
using H\"older's inequality~\eqref{e:holder-special}
$$
\begin{aligned}
\|f\|_{L^2}^2&\leq \sum_{\ell\in\mathbb Z\colon |\ell|\leq\alpha_1}
\|\hat f_\ell\|_{L^2(-1,1)}^2\\
&\leq CK^{21}\Big(\sum_\ell\|\indic_{U'}f_\ell\|_{H^{-10}}^2
+\exp\big(-C^{-1}\theta(K)K\big)\sum_\ell\|f_\ell\|_{H^{-10}}^2\Big)^{\kappa}\\
&\quad\cdot \Big(\sum_\ell\|f_\ell\|_{H^{-10}}^2\Big)^{1-\kappa}.
\end{aligned}
$$
Since
\begin{equation}
  \label{e:ell-sum}
\sum_\ell \|f_\ell\|_{H^{-10}}^2\leq C\|f\|_{L^2}^2,\quad
\sum_\ell \|\indic_{U'}f_\ell\|_{H^{-10}}^2\leq C\|f\|_{L^2(U')}^2
\end{equation}
and by the Minkowski inequality $(a+b)^\kappa\leq a^\kappa+b^\kappa$, $a,b\geq 0$,
we have
$$
\|f\|_{L^2}^2\leq CK^{21}\|f\|_{L^2(U')}^{2\kappa}\cdot\|f\|_{L^2}^{2(1-\kappa)}
+CK^{21}\exp\big(-C^{-1}\theta(K)\kappa K\big)\|f\|_{L^2}^2.
$$
Recalling that $\kappa=\exp(-C/\theta(K))$, $\delta<1$,
and the definition~\eqref{e:theta-def} of~$\theta(K)$
we have $\kappa K\geq C^{-1}\sqrt{K}$ and thus
$$
\lim_{K\to \infty}K^{21}\exp\big(-C^{-1}\theta(K)\kappa K\big)=0.
$$
Therefore, fixing $K$ large enough depending only on $\delta,C_R,c_1$ we have
$$
\|f\|_{L^2}^2\leq CK^{21}\|f\|_{L^2(U')}^{2\kappa}\cdot \|f\|_{L^2}^{2(1-\kappa)}
$$
which implies~\eqref{e:step} with $c_3=(CK^{21})^{-{1\over 2\kappa}}$.
\end{proof}
%%%%%%%%%%%%%%%%%%%%%%%%%%%%%%%%%%%%%%%%%%%%%%%%%%%%%%%%%%%%%%%%%%%%%%%%%%%%%%%%

%%%%%%%%%%%%%%%%%%%%%%%%%%%%%%%%%%%%%%%%%%%%%%%%%%%%%%%%%%%%%%%%%%%%%%%%%%%%%%%%
\subsection{The iteration argument}
  \label{s:iteration-argument}

We now finish the proof of Theorem~\ref{t:general-fup} by iterating
Proposition~\ref{l:step}. Let $\delta,C_R,N,X,Y$ satisfy
the assumptions of Theorem~\ref{t:general-fup}.

First of all, Lemma~\ref{l:regular-lebesgue}
gives the Lebesgue measure bounds
$$
\mu_L(X)\leq 24C_R^2N^{\delta-1},\quad
\mu_L(Y)\leq 24C_R^2N^\delta.
$$
Applying H\"older's inequality twice and using~\eqref{e:inverse-fourier}, we see that for
each $f\in L^2(\mathbb R)$ with $\supp\hat f\subset Y$
\begin{equation}
\begin{aligned}
\|f\|_{L^2(X)}
&\ \leq\
\sqrt{\mu_L(X)}\|f\|_{L^\infty}
\ \leq\
\sqrt{\mu_L(X)}\|\hat f\|_{L^1}\\
&\ \leq\
\sqrt{\mu_L(X)\mu_L(Y)}\|\hat f\|_{L^2}
\ \leq\
24C_R^2 N^{\delta-{1\over 2}}\|f\|_{L^2}
\end{aligned}
\end{equation}
where we used the Lebesgue measure to define $\|f\|_{L^2(X)}$.
This implies~\eqref{e:general-fup} for $\delta<1/2$ with $\beta=1/2-\delta$.
Therefore, we henceforth assume that $1/2\leq \delta<1$
(though we will only use that $0<\delta<1$).

Put
$$
L:=\big\lceil (3C_R)^{2\over 1-\delta}\big\rceil\in\mathbb N
$$
so that~\eqref{e:missingint} holds.
Let $V_n(X)$, $n\in\mathbb Z$, be the elements of the tree of intervals
covering $X$ constructed in~\eqref{e:covering-tree}.
Because of our choice of $L$
the tree $V_n(X)$ satisfies the missing child property,
Lemma~\ref{l:missing-child}, which will be used in the proof of Lemma~\ref{l:step-weight} below.
Define the coarse-graining of $X$ on the scale $L^{-n}$
\begin{equation}
  \label{e:U-n-def}
U_n:=\bigcup_{I\in V_n(X)}I\Big({1\over 10 L^{n}}\Big)\ \supset\ X\Big({1\over 10 L^{n}}\Big).
\end{equation}
Here we use the notation~\eqref{e:nbhd} for neighborhoods of sets.

We use the sets $U_n$ to construct a family of weights. Let $\varphi$ be a nonnegative
Schwartz function such that
$$
\supp\widehat\varphi\subset [-1,1],\quad
\int_{\mathbb R}\varphi(x)\,dx=1,
$$
and define for $n\in\mathbb Z$
$$
\varphi_n(x):=L^n\cdot \varphi(L^n x),\quad
\widehat\varphi_n(\xi)=\widehat\varphi(L^{-n}\xi).
$$
Take $T\in\mathbb N$
and define for $n\in\mathbb Z$ the following weight (see Figure~\ref{f:iteration-weight}):
$$
\Psi_n:=\mathbf 1_{U_{n+1}}*\varphi_{n+T}.
$$
We will later fix $T$ independently of $N$, see~\eqref{e:T-fixed},
and take $0\leq n\lesssim \log N$.
Note that $\Psi_n$ is a Schwartz function and $0\leq\Psi_n\leq 1$.

The fattening of the intervals in the definition of $U_n$ and
the need for the parameter $T$ are explained by the following
lemma which is used at the end of the proof:
%%%%%%%%%%%%%%%%%%%%%%%%%%%%%%%%%%%%%%%%%%%%%%%%%%%%%%%%%%%%%%%%%%%%%%%%%%%%%%%%
\begin{lemm}
  \label{l:Psi-lower}
There exists a constant $C_\varphi$ depending only on $\varphi$ such that for all~$n$
\begin{equation}
  \label{e:Psi-lower}
\Psi_n\geq 1-{C_\varphi\over L^{T-1}}\quad\text{on }X.
\end{equation}
\end{lemm}
%%%%%%%%%%%%%%%%%%%%%%%%%%%%%%%%%%%%%%%%%%%%%%%%%%%%%%%%%%%%%%%%%%%%%%%%%%%%%%%%
\begin{proof}
Let $x\in X$. Then by~\eqref{e:U-n-def}
$$
\Big[x-{1\over 10L^{n+1}},x+{1\over 10L^{n+1}}\Big]\ \subset\
 U_{n+1}.
$$
We have
$$
\Psi_n(x)=\int_{\mathbb R}\mathbf 1_{U_{n+1}}(x-L^{-n-T}y)\varphi(y)\,dy
\geq \int_{-L^{T-1}/10}^{L^{T-1}/10} \varphi(y)\,dy
$$
and~\eqref{e:Psi-lower} follows since $\varphi$ is a Schwartz function
of integral 1.
\end{proof}
%%%%%%%%%%%%%%%%%%%%%%%%%%%%%%%%%%%%%%%%%%%%%%%%%%%%%%%%%%%%%%%%%%%%%%%%%%%%%%%%
Next, Proposition~\ref{l:step} implies that when $\supp \hat f\subset Y(2L^n)$,
a positive proportion of the $L^2$ mass of $f$ is removed
when multiplying by the weight $\Psi_n$ (this is similar to restricting
$f$ to $U_{n+1}$, which is the coarse-graining of $X$ on the scale $L^{-n-1}$):
%%%%%%%%%%%%%%%%%%%%%%%%%%%%%%%%%%%%%%%%%%%%%%%%%%%%%%%%%%%%%%%%%%%%%%%%%%%%%%%%
\begin{lemm}
  \label{l:step-weight}
There exists $\tau>0$ depending only on $\delta,C_R$ such that
for all $T\in\mathbb N$,
$n\in\mathbb N$
such that $L^{n+1}\leq N$ and
$$
f\in L^2(\mathbb R),\quad
\supp\hat f\subset Y(2L^n),
$$
we have $\supp\widehat{\Psi_n f}\subset Y(2L^{n+T})$
and
\begin{equation}
  \label{e:step-weight}
\|\Psi_n f\|_{L^2}\leq (1-\tau) \|f\|_{L^2}.
\end{equation}
\end{lemm}
%%%%%%%%%%%%%%%%%%%%%%%%%%%%%%%%%%%%%%%%%%%%%%%%%%%%%%%%%%%%%%%%%%%%%%%%%%%%%%%%
\begin{proof}
By~\eqref{e:convolution-fourier}, we have
$\supp\widehat\Psi_n\subset \supp\widehat\varphi_{n+T}\subset [-L^{n+T},L^{n+T}]$.
Since $\widehat{\Psi_n f}=\widehat\Psi_n *\hat f$,
$$
\supp\widehat{\Psi_n f}\ \subset\ \supp \hat f+[-L^{n+T},L^{n+T}]
$$
which gives the Fourier support condition on $\Psi_{n} f$.

%%%%%%%%%%%%%%%%%%%%%%%%%%%%%%%%%%%%%%%%%%%%%%%%%%%%%%%%%%%%%%%%%%%%%%%%%%%%%%%%
\begin{figure}
\includegraphics{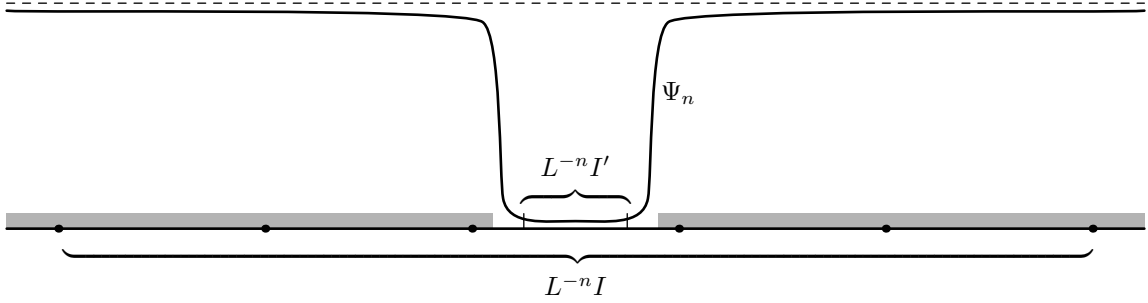}
\caption{The weight $\Psi_n$ for large $T$ and the interval $L^{-n}I'$ used in the proof
of Lemma~\ref{l:step-weight}. The dashed line corresponds to the
constant function~1 and the dots mark $L^{-n-1}\mathbb Z$;
the shaded region is $U_{n+1}$.}
\label{f:iteration-weight}
\end{figure}
%%%%%%%%%%%%%%%%%%%%%%%%%%%%%%%%%%%%%%%%%%%%%%%%%%%%%%%%%%%%%%%%%%%%%%%%%%%%%%%%

It remains to show~\eqref{e:step-weight}. Define the following
rescaling of $f$:
$$
\widetilde f(x):=L^{-n/2}\cdot f(L^{-n} x).
$$
Then $\|\widetilde f\|_{L^2( A)}=\|f\|_{L^2(L^{-n}\cdot A)}$ for
any set $A$ and, with $\mathcal F$ denoting the Fourier transform,
$$
\supp \mathcal F\widetilde f\ \subset\ \widetilde Y:={1\over L^n}Y+[-2,2]\ \subset\ [-\alpha_1,\alpha_1],\quad
\alpha_1:={10N\over  L^n}.
$$
By Lemmas~\ref{l:regular-scale}--\ref{l:regular-fatten} the set $\widetilde Y$ is
$\delta$-regular with constant $1000C_R$ on scales 1 to $\alpha_1$.

Let $\mathcal I$ be the partition of $\mathbb R$ into size 1 intervals
defined in~\eqref{e:cal-I}. For each $I\in\mathcal I$,
choose an interval $I'\subset I$ of size $c_1:=(2L)^{-1}$ as follows.
There exists $j\in \mathbb Z$ such that
$I_j:=L^{-1}[j,j+1]$ is contained in $I$ and satisfies $L^{-n}I_j\notin V_{n+1}(X)$.
Indeed, for $L^{-n}I\notin V_n(X)$ this is obvious (as one
can take any $I_j$ contained in $I$) and for $L^{-n}I\in V_n(X)$
it follows by Lemma~\ref{l:missing-child}.
We then let $I'\subset I_j$ have the same center as~$I_j$
and size $|I'|=c_1={1\over 2}|I_j|$. Note that 
the intervals $L^{-n}I'$ are relatively far from~$X$, more precisely
(see Figure~\ref{f:iteration-weight})
\begin{equation}
  \label{e:interval-far}
L^{-n}I'\,\cap\, U_{n+1}\Big({1\over 10L^{n+T}}\Big)=\emptyset.
\end{equation}

Applying Proposition~\ref{l:step} to the function $\widetilde f$,
the set $\widetilde Y$, and the subintervals $I'$ described in the
previous paragraph, we obtain for some $c_3>0$ depending
only on $\delta,C_R$
$$
\|\widetilde f\|_{L^2(U')}\geq c_3\|\widetilde f\|_{L^2},\quad
U':=\bigcup_{I\in\mathcal I}I',
$$
therefore 
\begin{equation}
  \label{e:f-kindof}
\|f\|_{L^2(L^{-n}\cdot U')}\geq c_3\|f\|_{L^2}.
\end{equation}
By~\eqref{e:interval-far} we have
$$
\Psi_n\leq c_\varphi\quad\text{on  }L^{-n}U',\quad
c_\varphi:=\int_{\mathbb R\setminus \left[-{1\over 10},{1\over 10}\right]} \varphi(x)\,dx<1.
$$
Since $0\leq \Psi_n\leq 1$, together with~\eqref{e:f-kindof} this implies
$$
\begin{aligned}
\|\Psi_n f\|_{L^2}^2&\leq c_\varphi^2\|f\|_{L^2(L^{-n}\cdot U')}^2
+\|f\|_{L^2(\mathbb R\setminus (L^{-n}\cdot U'))}^2
\\
&=\|f\|_{L^2}^2-(1-c_\varphi^2)\|f\|_{L^2(L^{-n}\cdot U')}^2
\\
&\leq \|f\|_{L^2}^2-(1-c_\varphi^2)c_3^2\|f\|_{L^2}^2. 
\end{aligned}
$$
This gives~\eqref{e:step-weight} where $\tau$ is defined by
$(1-\tau)^2=1-(1-c_\varphi^2)c_3^2$.
\end{proof}
%%%%%%%%%%%%%%%%%%%%%%%%%%%%%%%%%%%%%%%%%%%%%%%%%%%%%%%%%%%%%%%%%%%%%%%%%%%%%%%%
We are now ready to finish the proof of Theorem~\ref{t:general-fup}.
To do this we iterate Lemma~\ref{l:step-weight} $\sim\log N$ times.
At each next step we use coarser information on frequency
(that is, the Fourier support $\supp\hat f_m$ is contained in larger neighborhoods of $Y$)
and finer information on position
(that is, $f_m$ involves cutoffs to smaller neighborhoods of $X$).
%%%%%%%%%%%%%%%%%%%%%%%%%%%%%%%%%%%%%%%%%%%%%%%%%%%%%%%%%%%%%%%%%%%%%%%%%%%%%%%%
\begin{proof}[Proof of Theorem~\ref{t:general-fup}]
Assume that $f\in L^2(\mathbb R)$ and $\supp\hat f\subset Y$.
Fix large $T\in\mathbb N$ to be chosen below. For $m\in \mathbb N$,
define
$$
f_m:=\bigg(\prod_{\ell=1}^{m-1}\Psi_{\ell T}\bigg) f.
$$
Iterating Lemma~\ref{l:step-weight}, we see that
for each $m\in \mathbb N$ such that $L^{(m-1)T+1}\leq N$, we have
$$
\|f_m\|_{L^2}\leq (1-\tau)^{m-1}\|f\|_{L^2},\quad
\supp \hat f_m\subset Y(2L^{mT}).
$$
Then by Lemma~\ref{l:Psi-lower}, for all $m\in \mathbb N$ such that
$L^{(m-1)T+1}\leq N$
\begin{equation}
  \label{e:almost-there}
\|f\|_{L^2(X)}\leq
\Big(1-{C_\varphi\over L^{T-1}}\Big)^{1-m}\|f_m\|_{L^2}\leq
\Big(1-{C_\varphi\over L^{T-1}}\Big)^{1-m}(1-\tau)^{m-1}\|f\|_{L^2}.
\end{equation}
Fix $T$ large enough depending only on $\delta,C_R$ so that
\begin{equation}
  \label{e:T-fixed}
\Big(1-{C_\varphi\over L^{T-1}}\Big)^{-1}(1-\tau)\leq 1-{\tau\over 2}.
\end{equation}
Then~\eqref{e:almost-there} gives
$$
\|f\|_{L^2(X)}\leq \Big(1-{\tau\over 2}\Big)^{m-1}\|f\|_{L^2}.
$$
Taking $m$ such that $L^{(m-1)T+1}\leq N\leq L^{mT+1}$,
we get~\eqref{e:general-fup}
with
$$
\beta=-{\log\big(1-{\tau\over 2}\big)\over T\log L},
$$
finishing the proof.
\end{proof}
%%%%%%%%%%%%%%%%%%%%%%%%%%%%%%%%%%%%%%%%%%%%%%%%%%%%%%%%%%%%%%%%%%%%%%%%%%%%%%%%

%%%%%%%%%%%%%%%%%%%%%%%%%%%%%%%%%%%%%%%%%%%%%%%%%%%%%%%%%%%%%%%%%%%%%%%%%%%%%%%%
%%%%%%%%%%%%%%%%%%%%%%%%%%%%%%%%%%%%%%%%%%%%%%%%%%%%%%%%%%%%%%%%%%%%%%%%%%%%%%%%
\section{Hyperbolic fractal uncertainty principle}
  \label{s:special-fup}

In this section, we generalize Theorem~\ref{t:general-fup} first by allowing
a variable amplitude (\S\ref{s:pseudo}) and then
by taking a general phase (\S\ref{s:fio}). Both generalizations are
stated using the semiclassical parameter $h>0$ corresponding to the
inverse of the frequency. In~\S\ref{s:special-fup-proof},
we apply the result of~\S\ref{s:fio} to prove Theorem~\ref{t:special-fup}.

%%%%%%%%%%%%%%%%%%%%%%%%%%%%%%%%%%%%%%%%%%%%%%%%%%%%%%%%%%%%%%%%%%%%%%%%%%%%%%%%
\subsection{Uncertainty principle with variable amplitude}
  \label{s:pseudo}

We first prove the following semiclassical rescaling of Theorem~\ref{t:general-fup}
which also relaxes the assumptions on the sets $X,Y$.
In particular it allows for unbounded $X,Y$ but takes their
intersections with bounded intervals, which is a more convenient
assumption for applications.
Recall the notation~\eqref{e:nbhd} for neighborhoods of sets and
the semiclassical Fourier transform~\eqref{e:F-h}.
%%%%%%%%%%%%%%%%%%%%%%%%%%%%%%%%%%%%%%%%%%%%%%%%%%%%%%%%%%%%%%%%%%%%%%%%%%%%%%%%
\begin{prop}
  \label{l:semi-fup}
Let $0\leq \delta<1$, $C_R,C_I\geq 1$, and assume
that $X,Y\subset \mathbb R$ are $\delta$-regular with constant $C_R$
on scales $0$ to $1$. Let $I_X,I_Y\subset\mathbb R$ be intervals
with $|I_X|,|I_Y|\leq C_I$.
Then there exists $\beta>0$ depending only on~$\delta,C_R$
and $C>0$ depending only on $\delta,C_R,C_I$
such that for all $h\in (0,1)$
\begin{equation}
  \label{e:semi-fup}
\|\indic_{X(h)\cap I_X} \mathcal F_h^*\indic_{Y(h)\cap I_Y}\|_{L^2(\mathbb R)\to L^2(\mathbb R)}
\leq Ch^\beta.
\end{equation}
\end{prop}
%%%%%%%%%%%%%%%%%%%%%%%%%%%%%%%%%%%%%%%%%%%%%%%%%%%%%%%%%%%%%%%%%%%%%%%%%%%%%%%%
\begin{proof}
Without loss of generality, we may assume that $h$ is small
depending on $\delta,C_R$.
By Lemma~\ref{l:regular-fatten}, the sets $X(h),Y(h)$ are
$\delta$-regular with constant $8C_R$ on scales $h$ to 1.
By Lemma~\ref{l:regular-split} there exist collections
of disjoint intervals $\mathcal J_X,\mathcal J_Y$ such that
$$
\begin{gathered}
X(h)=\bigsqcup_{J\in\mathcal J_X} X_J,\quad
X_J:=X(h)\cap J;\\
Y(h)=\bigsqcup_{J'\in\mathcal J_Y} Y_{J'},\quad
Y_{J'}:=Y(h)\cap J';\\
(32C_R)^{-{2\over 1-\delta}}\leq |J|\leq 1\quad\text{for all }J\in \mathcal J_X\cup\mathcal J_Y,
\end{gathered}
$$
and the sets $X_J,Y_{J'}$ are $\delta$-regular with constant $\widetilde C_R:=(100C_R)^{2\over 1-\delta}C_R$ on scales $h$ to 1.

We have the following estimate for each $J\in \mathcal J_X$, $J'\in\mathcal J_Y$,
where $\beta,C>0$ depend only on $\delta,C_R$:
\begin{equation}
  \label{e:semifup-1}
\|\indic_{X_J}\mathcal F_h^*\indic_{Y_{J'}}\|_{L^2\to L^2}\leq Ch^\beta.
\end{equation}
Indeed, since
$X_J,Y_{J'}$ have diameter no more than 1,
we may shift $X_J,Y_{J'}$ to make them lie inside $[-1,1]$.
By~\eqref{e:shifted-fup} this does not change the left-hand side of~\eqref{e:semifup-1};
by Lemma~\ref{l:regular-scale} it does not
change $\delta$-regularity.
Take arbitrary $g\in L^2(\mathbb R)$
and put $f:=\mathcal F_h^* \indic_{Y_{J'}} g$ and $N:=h^{-1}$.
Then $\supp\hat f$ lies in $N\cdot Y_{J'}$,
which by Lemma~\ref{l:regular-scale} is $\delta$-regular with
constant~$\widetilde C_R$ on scales 1 to $N$. 
Applying
Theorem~\ref{t:general-fup}, we obtain
$$
\|\indic_{X_J}\mathcal F_h^*\indic_{Y_{J'}}g\|_{L^2}\ =\
\|\indic_{X_J} f\|_{L^2}\ \leq\ Ch^\beta\|f\|_{L^2}\ \leq\
Ch^\beta\|g\|_{L^2}
$$
implying~\eqref{e:semifup-1}.

Next, the number of intervals in $\mathcal J_X$ intersecting $I_X$ is bounded as follows:
$$
\#\{J\in\mathcal J_X\mid J\cap I_X\neq\emptyset\}\ \leq\ (32C_R)^{2\over 1-\delta}C_I+2
$$
and a similar estimate holds for the number of intervals in $\mathcal J_Y$ intersecting $I_Y$.
Combining these estimates with~\eqref{e:semifup-1} and using the triangle
inequality~\eqref{e:fup-splitting}, we obtain~\eqref{e:semi-fup},
finishing the proof.
\end{proof}
%%%%%%%%%%%%%%%%%%%%%%%%%%%%%%%%%%%%%%%%%%%%%%%%%%%%%%%%%%%%%%%%%%%%%%%%%%%%%%%%
We now prove a fractal uncertainty principle for operators
$A=A(h):L^2(\mathbb R)\to L^2(\mathbb R)$ of the form
\begin{equation}
  \label{e:pseudo-form}
A f(x)=h^{-1/2}\int_{\mathbb R} e^{2\pi ix\xi/h}a(x,\xi)f(\xi)\,d\xi
\end{equation}
where $a(x,\xi)\in C_0^\infty(\mathbb R^2)$ satisfies for each $k$
and some constants $C_k,C_a$
\begin{equation}
  \label{e:symbol-estimate}
\sup |\partial^k_x a|\leq C_{k},\quad
\diam \supp a\leq C_a.
\end{equation}
In the statement below it is convenient to replace neighborhoods of size $h$
by those of size $h^\rho$ where $\rho\in (0,1)$. In practice we will
take $\rho$ very close to 1 so that the resulting losses do not negate
the gain $h^\beta$.
The proof of Proposition~\ref{l:upgraded-fup} relies on Proposition~\ref{l:semi-fup}
and the fact that for $\rho<1$, functions in the range
of $A(h)\indic_{Y(h^\rho)}$ are concentrated
on $Y(2h^\rho)$ in the semiclassical Fourier space.
%%%%%%%%%%%%%%%%%%%%%%%%%%%%%%%%%%%%%%%%%%%%%%%%%%%%%%%%%%%%%%%%%%%%%%%%%%%%%%%%
\begin{prop}
  \label{l:upgraded-fup}
Let $0\leq \delta<1$, $C_R\geq 1$ and assume
that $X,Y\subset \mathbb R$ are $\delta$-regular with constant $C_R$
on scales $0$ to $1$ and \eqref{e:symbol-estimate} holds.
Then there exists $\beta>0$ depending only on $\delta,C_R$
such that for all $\rho\in (0,1)$ and $h\in (0,1)$
\begin{equation}
  \label{e:upgraded-fup}
\|\indic_{X(h^\rho)} A(h)\indic_{Y(h^\rho)}\|_{L^2(\mathbb R)\to L^2(\mathbb R)}\leq Ch^{\beta-2(1-\rho)}
\end{equation}
where the constant $C$ depends only on $\delta,C_R,\{C_{k}\},C_a,\rho$.
\end{prop}
%%%%%%%%%%%%%%%%%%%%%%%%%%%%%%%%%%%%%%%%%%%%%%%%%%%%%%%%%%%%%%%%%%%%%%%%%%%%%%%%
\begin{proof}
Denote by
$C$ constants which depend only on
$\delta,C_R,\{C_{k}\},C_a,\rho$ (the value of $C$ may differ
in different parts of the proof).
We note that%
\footnote{If all $(x,\xi)$-derivatives of $a$ are bounded,
then $\mathcal F_hA$ is a pseudodifferential operator
and~\eqref{e:psido-bdd} follows from the Calder\'on--Vaillancourt Theorem.} 
\begin{equation}
  \label{e:psido-bdd}
\|A\|_{L^2\to L^2}\leq C.
\end{equation}
To see this, we compute the integral kernel of $A^*A$:
$$
\mathcal K_{A^*A}(\xi,\eta)=h^{-1}\int_{\mathbb R}e^{2\pi ix(\eta-\xi)/h}\overline{a(x,\xi)}
a(x,\eta)\,dx.
$$
Using~\eqref{e:symbol-estimate} and repeated integration by parts in $x$
we obtain
$$
|\mathcal K_{A^*A}(\xi,\eta)|\leq Ch^{-1}\Big\langle {\xi-\eta\over h}\Big\rangle^{-10}
$$
which by Schur's inequality (see for instance~\cite[Theorem~4.21]{e-z}) gives $\|A^*A\|_{L^2\to L^2}\leq C$ and thus~\eqref{e:psido-bdd} holds.

Take intervals $I_X,I_Y$ such that $\supp a\subset I_X\times I_Y$ and
$|I_X|,|I_Y|\leq C_a$.
We write
$$
\begin{gathered}
\indic_{X(h^\rho)}A\indic_{Y(h^\rho)}
=\indic_{X(h^\rho)\cap I_X}A\indic_{Y(h^\rho)\cap I_Y}
=\indic_{X(h^\rho)\cap I_X}\mathcal F_h^*A_1+A_2\mathcal F_h A\indic_{Y(h^\rho)\cap I_Y},\\
A_1:=\indic_{\mathbb R\setminus (Y(2h^\rho)\cap I_Y(1))}\mathcal F_hA\indic_{Y(h^\rho)\cap I_Y},\quad
A_2:=\indic_{X(h^\rho)\cap I_X}\mathcal F_h^* \indic_{Y(2h^\rho)\cap I_Y(1)},
\end{gathered}
$$
so that by~\eqref{e:psido-bdd}
\begin{equation}
  \label{e:upgrade-0}
\|\indic_{X(h^\rho)}A\indic_{Y(h^\rho)}\|_{L^2\to L^2}\leq
\|A_1\|_{L^2\to L^2}+C\|A_2\|_{L^2\to L^2}.
\end{equation}
The operator $\mathcal F_hA$ is pseudodifferential,
thus its integral kernel is rapidly decaying once we step
$h^\rho$ away from the diagonal.
Since the sets $\mathbb R\setminus (Y(2h^\rho)\cap I_Y(1))$ and $Y(h^\rho)\cap I_Y$ are distance
$h^\rho$ away from each other, this implies
\begin{equation}
  \label{e:upgrade-1}
\|A_1\|_{L^2\to L^2}\leq Ch^{10}.
\end{equation}
More precisely, to show~\eqref{e:upgrade-1} we compute the integral kernel of $A_1$:
$$
\mathcal K_{A_1}(\xi,\eta)=\mathbf 1_{\mathbb R\setminus (Y(2h^\rho)\cap I_Y(1))}(\xi)\mathbf 1_{Y(h^\rho)\cap I_Y}(\eta)\cdot h^{-1}\int_{\mathbb R}e^{2\pi ix(\eta-\xi)/h}a(x,\eta)\,dx.
$$
Note that $|\xi-\eta|\geq h^\rho$
on $\supp \mathcal K_{A_1}$.
Using~\eqref{e:symbol-estimate} and repeated integration by parts in~$x$, we obtain
for each $M\in\mathbb N_0$
$$
|\mathcal K_{A_1}(\xi,\eta)|\leq C_M h^{-1}\Big\langle{\xi-\eta\over h}\Big\rangle^{-M-1}
$$
which implies~\eqref{e:upgrade-1} by another application of Schur's inequality
as soon as $M\geq {10\over 1-\rho}$.

We now estimate $\|A_2\|$. By Proposition~\ref{l:semi-fup}
there exists $\beta>0$ depending only on $\delta,C_R$ such that
$$
\|\indic_{X(h)\cap I_X(1)}\mathcal F_h^*\indic_{Y(h)\cap I_Y(2)}\|_{L^2\to L^2}\leq Ch^\beta.
$$
We cover $X(h^\rho)\cap I_X$, $Y(2h^\rho)\cap I_Y(1)$ as follows:
$$
X(h^\rho)\cap I_X\subset \bigcup_{p\in h\mathbb Z\atop |p|\leq h^{\rho}}
\big(X(h)\cap I_X(1)\big)+p,\quad
Y(2h^\rho)\cap I_Y(1)\subset \bigcup_{q\in h\mathbb Z\atop |q|\leq 2h^{\rho}}
\big(Y(h)\cap I_Y(2)\big)+q.
$$
Each of the above unions has at most $10h^{\rho-1}$ elements,
therefore by~\eqref{e:shifted-fup}
and the triangle inequality~\eqref{e:fup-splitting} we get
\begin{equation}
  \label{e:upgrade-2}
\|A_2\|_{L^2\to L^2}\leq Ch^{\beta-2(1-\rho)}.
\end{equation}
Combining~\eqref{e:upgrade-0}--\eqref{e:upgrade-2}, we obtain~\eqref{e:upgraded-fup}.
\end{proof}
%%%%%%%%%%%%%%%%%%%%%%%%%%%%%%%%%%%%%%%%%%%%%%%%%%%%%%%%%%%%%%%%%%%%%%%%%%%%%%%%

%%%%%%%%%%%%%%%%%%%%%%%%%%%%%%%%%%%%%%%%%%%%%%%%%%%%%%%%%%%%%%%%%%%%%%%%%%%%%%%%
\subsection{Uncertainty principle with general phase}
  \label{s:fio}

We next prove a fractal uncertainty principle for operators
$B=B(h):L^2(\mathbb R)\to L^2(\mathbb R)$
of the form
\begin{equation}
  \label{e:fio-form}
Bf(x)=h^{-1/2}\int_{\mathbb R} e^{i\Phi(x,y)/h}b(x,y)\,f(y)\,dy
\end{equation}
where for some open set $U\subset\mathbb R^2$,
\begin{equation}
  \label{e:phase-symbol}
\Phi\in C^\infty(U;\mathbb R),\quad
b\in C_0^\infty(U),\quad
\partial^2_{xy}\Phi\neq 0\quad\text{on }U.
\end{equation}
The condition $\partial^2_{xy}\Phi\neq 0$ ensures that
locally we can write the graph of the twisted gradient of $\Phi$
in terms of some symplectomorphism $\varkappa$ of open subsets
of $T^*\mathbb R$:
\begin{equation}
  \label{e:varkappa-formula}
(x,\xi)=\varkappa(y,\eta)
\quad\Longleftrightarrow\quad
\xi=\partial_x\Phi(x,y),\
\eta=-\partial_y\Phi(x,y).
\end{equation}
Then $B$ is a Fourier integral operator associated to $\varkappa$,
see for instance~\cite[\S2.2]{hgap}. Note that symplectomorphisms
of the form~\eqref{e:varkappa-formula} satisfy the following
transversality condition: each vertical leaf $\{y=\const\}\subset T^*\mathbb R^2$ is
mapped by $\varkappa$ to a curve which is transversal to all vertical leaves $\{x=\const\}$.
Proposition~\ref{l:fup-fio} below can be interpreted in terms of the theory of Fourier integral
operators, however we give a proof which is self-contained and does not explicitly rely on this theory.
%%%%%%%%%%%%%%%%%%%%%%%%%%%%%%%%%%%%%%%%%%%%%%%%%%%%%%%%%%%%%%%%%%%%%%%%%%%%%%%%
\begin{prop}
  \label{l:fup-fio}
Let $0\leq\delta <1$, $C_R\geq 1$ and assume that
$X,Y\subset \mathbb R$ are $\delta$-regular with constant $C_R$
on scales $0$ to $1$ and \eqref{e:phase-symbol}
holds. Then there exist $\beta>0$, $\rho\in (0,1)$ depending only on $\delta,C_R$
and  $C>0$ depending only on $\delta,C_R,\Phi,b$
such that for all $h\in (0,1)$
\begin{equation}
  \label{e:fup-fio}
\|\indic_{X(h^\rho)} B(h)\indic_{Y(h^\rho)}\|_{L^2(\mathbb R)\to L^2(\mathbb R)}\leq Ch^\beta.
\end{equation}
\end{prop}
%%%%%%%%%%%%%%%%%%%%%%%%%%%%%%%%%%%%%%%%%%%%%%%%%%%%%%%%%%%%%%%%%%%%%%%%%%%%%%%%
\Remark
The value of $\beta$ in Proposition~\ref{l:fup-fio} (and in Theorem~\ref{t:special-fup}) is smaller than the one
in Theorem~\ref{t:general-fup} and Propositions~\ref{l:semi-fup}--\ref{l:upgraded-fup}.
Denoting the latter by $\tilde\beta$, our argument
gives~\eqref{e:fup-fio} with $\beta=\tilde\beta/4$~-- see~\eqref{e:rho-fixed} below.
By taking $\rho$ sufficiently close to~1, one can get any $\beta<\tilde\beta/2$.
However, since we do not specify the value of $\beta$ this difference
is irrelevant to the final result.

We first note that it is enough to prove Proposition~\ref{l:fup-fio}
under the assumption
\begin{equation}
  \label{e:Phi-extra}
1\ <\ |\partial^2_{xy}\Phi|\ <\ 2\quad\text{on }U.
\end{equation}
Indeed, assume that Proposition~\ref{l:fup-fio} is established
for all $\Phi$ satisfying~\eqref{e:Phi-extra}. Then
it also holds for all $\lambda_\Phi>0$ and $\Phi$ satisfying
\begin{equation}
  \label{e:Phi-extra-2}
\lambda_\Phi\ <\ |\partial^2_{xy}\Phi|\ <\ 2\lambda_\Phi\quad\text{on }U
\end{equation}
where $\beta,\rho$ do not depend on $\lambda_\Phi$ but $C$ does. Indeed,
put $\widetilde\Phi:=\lambda_\Phi^{-1}\Phi$, then $\widetilde\Phi$
satisfies~\eqref{e:Phi-extra}. If $\widetilde B(h)$ is given by~\eqref{e:fio-form}
with $\Phi$ replaced by $\widetilde\Phi$, then
$B(h)=\widetilde B(\lambda_\Phi^{-1}h)$, thus by slightly increasing
$\rho$ we see that Proposition~\ref{l:fup-fio} for $\widetilde\Phi$ implies
it for $\Phi$. Finally, for the case of general $\Phi$ we
use a partition of unity for $b$ and shrink $U$ accordingly
to split $B$ into the sum of finitely many operators of the form~\eqref{e:fio-form}
each of which has a phase function satisfying~\eqref{e:Phi-extra-2}
for some $\lambda_\Phi$.

The proof of Proposition~\ref{l:fup-fio} relies on the following
statement which fattens the set $X$ by $h^{\rho/2}$, intersects $Y(h^\rho)$ with a
size $h^{1/2}$ interval, and is proved by making a change
of variables and taking the semiclassical
parameter $\tilde h:=h^{1/2}$ in Proposition~\ref{l:upgraded-fup}:
%%%%%%%%%%%%%%%%%%%%%%%%%%%%%%%%%%%%%%%%%%%%%%%%%%%%%%%%%%%%%%%%%%%%%%%%%%%%%%%%
\begin{lemm}
  \label{l:one-piece}
Assume~\eqref{e:Phi-extra} holds. Then there exist $\beta>0$, $\rho\in (0,1)$ depending only on $\delta,C_R$
and $C>0$ depending only on $\delta,C_R,\Phi,b$
such that for all
$h\in (0,1)$ and all intervals $J$ of size $h^{1/2}$
\begin{equation}
  \label{e:fup-fio-piece}
\|\indic_{X(h^{\rho/2})}B(h)\indic_{Y(h^\rho)\cap J}\|_{L^2\to L^2}\leq Ch^\beta.
\end{equation}
\end{lemm}
%%%%%%%%%%%%%%%%%%%%%%%%%%%%%%%%%%%%%%%%%%%%%%%%%%%%%%%%%%%%%%%%%%%%%%%%%%%%%%%%
\begin{proof}
Fix $\rho\in ({1\over 2},1)$ to be chosen later.
Breaking the symbol $b$ into pieces using a partition of unity,
we may assume that
$$
\supp b\ \subset\ I_X\times I_Y'\ \subset\ I_X\times I_Y\ \subset\  U
$$
where $I_X,I_Y,I'_Y$ are some intervals with $I'_Y\Subset I_Y$.
 We may assume that $J\subset I_Y$;
indeed, otherwise 
the operator in~\eqref{e:fup-fio-piece} is equal to 0
for $h$ small
enough.
Let $y_0$ be the center of $J$ and
define the function
$$
\varphi:I_X\to\mathbb R,\quad
\varphi(x)={1\over 2\pi}\partial_y\Phi(x,y_0).
$$
By~\eqref{e:Phi-extra} we have
\begin{equation}
  \label{e:phi-bound}
{1\over 2\pi}\ <\ |\partial_x\varphi|\ <\ {1\over \pi}\quad\text{on }I_X.
\end{equation}
In particular, $\varphi:I_X\to \varphi(I_X)$ is a diffeomorphism.
We extend $\varphi$ to a diffeomorphism of~$\mathbb R$ such that~\eqref{e:phi-bound}
holds on the entire $\mathbb R$.
Let $\Psi\in C^\infty(I_X\times I_Y)$ be the remainder in Taylor's formula
for $\Phi$, defined by
$$
\Phi(x,y)=\Phi(x,y_0)+2\pi(y-y_0)\varphi(x)+(y-y_0)^2\Psi(x,y),\quad
x\in I_X,\ y\in I_Y.
$$
Consider the isometries $W_X,W_Y:L^2(\mathbb R)\to L^2(\mathbb R)$ defined by
$$
W_Xf(x)=e^{-i\Phi(\varphi^{-1}(x),y_0)/h}\big|\partial_x(\varphi^{-1})(x)\big|^{1/2}f(\varphi^{-1}(x)),\quad
W_Yf(y)=h^{-1/4}f\Big({y-y_0\over h^{1/2}}\Big).
$$
Here we extend $\Phi(\varphi^{-1}(x),y_0)$ from $\varphi(I_X)$ to a real-valued
function on $\mathbb R$.
We also fix a function $\chi\in C_0^\infty((-1,1);[0,1])$ such that
$\chi=1$ near $[-{1\over 2},{1\over 2}]$ and define the cutoff $\chi_J$ by
$$
\chi_J(y)=\chi\Big({y-y_0\over h^{1/2}}\Big),\quad
\chi_J=1\quad\text{on }J.
$$
Put $A=A(h):=W_X B(h)\chi_J W_Y$, then we write $A$ in the form~\eqref{e:pseudo-form}:
$$
Af(x)=\tilde h^{-1/2}\int_{\mathbb R}e^{2\pi ix\xi/\tilde h}a(x,\xi;\tilde h)f(\xi)\,d\xi
$$
where $\tilde h:=h^{1/2}$ and
$$
a(x,\xi;\tilde h)=e^{i\xi^2\Psi(\varphi^{-1}(x),y_0+\tilde h\xi)}
\big|\partial_x(\varphi^{-1})(x)\big|^{1/2}b(\varphi^{-1}(x),y_0+\tilde h\xi)\chi(\xi).
$$
The amplitude $a$ satisfies~\eqref{e:symbol-estimate} with the constants
$C_k,C_a$ depending only on $\Phi,b$.
We now have
\begin{equation}
  \label{e:peacock}
\begin{aligned}
\|\indic_{X(h^{\rho/2})}B\indic_{Y(h^\rho)\cap J}\|_{L^2\to L^2}&
\leq
\|W_X\indic_{X(h^{\rho/2})}B\chi_J\indic_{Y(h^\rho)}W_Y\|_{L^2\to L^2}\\
&\leq
\|\indic_{\widetilde X(C_\Phi \tilde h^{\rho})}A\indic_{\widetilde Y(\tilde h^{2\rho-1})}\|_{L^2\to L^2}
\end{aligned}
\end{equation}
where $\widetilde X:=\varphi(X)$, $\widetilde Y:=h^{-1/2}(Y-y_0)$.
By Lemmas~\ref{l:regular-nonlinear} and~\ref{l:regular-expand-top}
the set $\widetilde X$ is $\delta$-regular with constant
$\widetilde C_R:=8\pi^2 C_R$ on scales 0 to 1.
By Lemma~\ref{l:regular-scale} the set $\widetilde Y$ has the same property.
Applying Proposition~\ref{l:upgraded-fup} we obtain
$$
\|\indic_{\widetilde X(\tilde h^{2\rho-1})}A\indic_{\widetilde Y(\tilde h^{2\rho-1})}\|_{L^2\to L^2}\leq
C\tilde h^{\tilde\beta-4(1-\rho)}=Ch^{{\tilde\beta\over 2}-2(1-\rho)}
$$
where $\tilde \beta>0$ depends only on $\delta,C_R$
and $C$ depends only on $\delta,C_R,\Phi,b,\rho$.
Fixing
\begin{equation}
  \label{e:rho-fixed}
\rho:=1-{1\over 8}\tilde\beta,\quad
\beta:={\tilde\beta\over 4},
\end{equation}
taking $h$ small enough so that $C_\Phi \tilde h^{\rho}\leq \tilde h^{2\rho-1}$,
and using~\eqref{e:peacock}, we obtain~\eqref{e:fup-fio-piece}.
\end{proof}
%%%%%%%%%%%%%%%%%%%%%%%%%%%%%%%%%%%%%%%%%%%%%%%%%%%%%%%%%%%%%%%%%%%%%%%%%%%%%%%%
We now finish the proof of Proposition~\ref{l:fup-fio} using almost orthogonality
similarly to~\cite[\S5.2]{hgap}:
%%%%%%%%%%%%%%%%%%%%%%%%%%%%%%%%%%%%%%%%%%%%%%%%%%%%%%%%%%%%%%%%%%%%%%%%%%%%%%%%
\begin{proof}[Proof of Proposition~\ref{l:fup-fio}]
Denote by $C$ constants which depend only on $\delta,C_R,\Phi,b$.
Since $\partial^2_{xy}\Phi\neq 0$ on $U$,
after using a partition of unity for $b$ and shrinking $U$ we may assume that
\begin{equation}
  \label{e:levels-apart}
|\partial_x\Phi(x,y)-\partial_x\Phi(x,y')|\geq C^{-1}|y-y'|\quad\text{for all }
(x,y),(x,y')\in U.
\end{equation}
Take $\beta>0$, $\rho\in (0,1)$ defined in Lemma~\ref{l:one-piece}.
By~\cite[Lemma~3.3]{hgap}, there exists $\psi=\psi(x;h)\in C^\infty(\mathbb R;[0,1])$
such that for some global constants $C_{k,\psi}$
\begin{gather}
\psi=1\quad\text{on }X(h^\rho),\quad
\supp\psi\subset X(h^{\rho/2});\\
  \label{e:psi-ders}
\sup |\partial^k_x\psi|\leq C_{k,\psi} h^{-\rho k/2}.
\end{gather}
Take the smallest interval $I_Y$ such that $\supp b\subset \mathbb R\times I_Y$.
Take a maximal set of ${1\over 2}h^{1/2}$-separated points
$$
y_1,\dots,y_N\in Y(h^\rho)\cap I_Y,\quad
N\leq Ch^{-1/2}
$$
and let $J_n$ be the interval of size $h^{1/2}$ centered at $y_n$.
Define the operators
$$
B_n:=\sqrt{\psi}\, B\indic_{Y(h^\rho)\cap J_n},\quad
n=1,\dots,N.
$$
Then by Lemma~\ref{l:one-piece} we have uniformly in $n$,
\begin{equation}
  \label{e:one-piece-used}
\|B_n\|_{L^2\to L^2}\ \leq\
\|\indic_{X(h^{\rho/2})}B\indic_{Y(h^\rho)\cap J_n}\|_{L^2\to L^2}\ \leq\
Ch^\beta.
\end{equation}
On the other hand, $Y(h^\rho)\cap I_Y\subset\bigcup_n (Y(h^\rho)\cap J_n)$ and thus
\begin{equation}
  \label{e:pieces-cover}
\|\indic_{X(h^\rho)}B\indic_{Y(h^\rho)}\|_{L^2\to L^2}\ \leq\
\big\|\sqrt{\psi}\, B\indic_{Y(h^\rho)\cap I_Y}\big\|_{L^2\to L^2}\ \leq\
\bigg\|\sum_{n=1}^N B_n\bigg\|_{L^2\to L^2}.
\end{equation}
We will estimate the right-hand side of~\eqref{e:pieces-cover}
by the Cotlar--Stein Theorem~\cite[Theorem~C.5]{e-z}.
We say that two points $y_n,y_m$ are \emph{close}
if $|y_n-y_m|\leq 10h^{1/2}$ and are \emph{far} otherwise.
Each point is close to at most 100 other points.
The following estimates hold when $y_n,y_m$ are far:
%%%%%%%%%%%%%%%%%%%%%%%%%%%%%%%%%%%%%%%%%%%%%%%%%%%%%%%%%%%%%%%%%%%%%%%%%%%%%%%%
\begin{gather}
  \label{e:cs-1}
B_nB_m^*=0,\\
  \label{e:cs-2}
\|B_n^*B_m\|_{L^2\to L^2}\leq Ch^{10}.
\end{gather}
%%%%%%%%%%%%%%%%%%%%%%%%%%%%%%%%%%%%%%%%%%%%%%%%%%%%%%%%%%%%%%%%%%%%%%%%%%%%%%%%
Indeed, \eqref{e:cs-1} follows immediately since $J_n\cap J_m=\emptyset$.
To show~\eqref{e:cs-2}, we compute the integral kernel
of $B_n^*B_m$:
$$
\mathcal K_{B_n^*B_m}(y,y')=\indic_{Y(h^\rho)\cap J_n}(y)
\indic_{Y(h^\rho)\cap J_m}(y')\cdot h^{-1}\int_{\mathbb R}
e^{{i\over h}\left(\Phi(x,y')-\Phi(x,y)\right)}\overline{b(x,y)}b(x,y')\psi(x)\,dx.
$$
Since $y_n,y_m$ are far, we have $|y-y'|\geq h^{1/2}$ on $\supp\mathcal K_{B_n^*B_m}$.
We now repeatedly integrate by parts in $x$.
Each integration produces a gain of $h^{1/2}$ due to~\eqref{e:levels-apart}
and a loss of $h^{-\rho/2}$ due to~\eqref{e:psi-ders}. Since $\rho<1$,
after finitely many steps
we obtain~\eqref{e:cs-2}. See the proof of~\cite[Lemma~5.2]{hgap}
for details.

Now~\eqref{e:one-piece-used}, \eqref{e:cs-1}, and~\eqref{e:cs-2} imply
by the Cotlar--Stein Theorem
$$
\bigg\|\sum_{n=1}^NB_n\bigg\|_{L^2\to L^2}\leq Ch^\beta
$$
which gives~\eqref{e:fup-fio} because of~\eqref{e:pieces-cover}.
\end{proof}
%%%%%%%%%%%%%%%%%%%%%%%%%%%%%%%%%%%%%%%%%%%%%%%%%%%%%%%%%%%%%%%%%%%%%%%%%%%%%%%%

%%%%%%%%%%%%%%%%%%%%%%%%%%%%%%%%%%%%%%%%%%%%%%%%%%%%%%%%%%%%%%%%%%%%%%%%%%%%%%%%
\subsection{Proof of Theorem~\ref{t:special-fup}}
  \label{s:special-fup-proof}

We parametrize the circle by $\theta\in \mathbb S^1:=\mathbb R/(2\pi \mathbb Z)$.
Let $\Lambda_\Gamma\subset\mathbb S^1$ be the limit set of $\Gamma$;
we lift it to a $2\pi$-periodic subset of $\mathbb R$,
denoted by $X$.

The set $X\subset\mathbb R$
is $\delta$-regular with some constant $C_R$ on scales 0 to 1,
where we can take as $\mu_X$ the Hausdorff measure of dimension $\delta$
or equivalently the lift of the Patterson--Sullivan measure~--
see for example~\cite[Theorem~7]{Sullivan} and~\cite[Lemma~14.13 and Theorem~14.14]{BorthwickBook}. Here $\delta\in [0,1]$ is the exponent of convergence
of Poincar\'e series of the group and $\delta<1$ when $M=\Gamma\backslash\mathbb H^2$ is convex co-compact but not compact~--
see for instance~\cite[\S2.5.2]{BorthwickBook} and~\cite[Theorem~2]{Beardon}.

Let $\mathcal B_\chi(h)$ be the operator defined in~\eqref{e:B-chi}.
By partition of unity, we may assume that $\supp\chi$ lies in the product
of two half-circles.
Then for all $h\in (0,1)$, $\rho\in (0,1)$
$$
\|\indic_{\Lambda_\Gamma(h^\rho)}\mathcal B_\chi(h)\indic_{\Lambda_\Gamma(h^\rho)}\|_{L^2(\mathbb S^1)\to L^2(\mathbb S^1)}
=\|\indic_{X(h^\rho)}B(h)\indic_{X(h^\rho)}\|_{L^2(\mathbb R)\to L^2(\mathbb R)}
$$
where $B=B(h)$ has the form~\eqref{e:fio-form}:
$$
Bf(\theta)=h^{-1/2}\int_{\mathbb R}
e^{i\Phi(\theta,\theta')/h}
b(\theta,\theta')f(\theta')\,d\theta'.
$$
Here, denoting
$U:=\{(\theta,\theta')\mid \theta-\theta'\notin 2\pi \mathbb Z\}$, the function
$b\in C_0^\infty(U)$ is a compactly
supported lift of $(2\pi)^{-1/2}\chi$ and
$\Phi\in C^\infty(U;\mathbb R)$ is given by
$$
\Phi(\theta,\theta')=\log 4+2\log\Big|\sin\Big({\theta-\theta'\over 2}\Big)\Big|,\quad
\theta,\theta'\in \mathbb R.
$$
We have
$$
\partial^2_{\theta\theta'}\Phi={1\over 2\sin^2\big({\theta-\theta'\over 2}\big)}\neq 0\quad\text{on }U.
$$
By Proposition~\ref{l:fup-fio} there exist $\beta>0$ and $\rho\in (0,1)$
depending only on $\delta,C_R$
and $C>0$ depending on $\delta,C_R,\chi$ such that for all $h\in (0,1)$,
$$
\|\indic_{X(h^\rho)}B(h)\indic_{X(h^\rho)}\|_{L^2\to L^2}\leq Ch^\beta
$$
which implies~\eqref{e:special-fup} and finishes the proof of Theorem~\ref{t:special-fup}.

%%%%%%%%%%%%%%%%%%%%%%%%%%%%%%%%%%%%%%%%%%%%%%%%%%%%%%%%%%%%%%%%%%%%%%%%%%%%%%%%
%%%%%%%%%%%%%%%%%%%%%%%%%%%%%%%%%%%%%%%%%%%%%%%%%%%%%%%%%%%%%%%%%%%%%%%%%%%%%%%%
\medskip\noindent\textbf{Acknowledgements.}
We would like to thank Maciej Zworski for several helpful discussions
about the spectral gap problem.
We are also grateful to two anonymous referees for numerous suggestions to improve
the manuscript.
JB is partially supported by NSF grant DMS-1301619.
SD is grateful for the hospitality of the Institute for Advanced Study during his
stay there in October 2016.
This research was conducted during the period SD served as
a Clay Research Fellow. 

%%%%%%%%%%%%%%%%%%%%%%%%%%%%%%%%%%%%%%%%%%%%%%%%%%%%%%%%%%%%%%%%%%%%%%%%%%%%%%%%

\end{document}